\documentclass[10pt,a4paper]{article}
\usepackage{amsmath}
\usepackage{amssymb}
\usepackage{amssymb}
\usepackage{cases}
\usepackage{bbm}
\usepackage{mathrsfs}
\usepackage{graphicx}
\usepackage{amsfonts}
\usepackage{enumerate}
\usepackage{theorem}
\usepackage[pdfstartview=FitH]{hyperref}
\makeatletter
\def\tank#1{\protected@xdef\@thanks{\@thanks
 \protect\footnotetext[0]{#1}}}
\def\bigfoot{

 \@footnotetext}
\makeatother

\newcommand{\ea}{\end{array}}

\allowdisplaybreaks

\newtheorem{theorem}{Theorem}[section]
\newtheorem{lem}{Lemma}[section]
\newtheorem{prp}[theorem]{Proposition}
\newtheorem{thm}[theorem]{Theorem}
\newtheorem{cor}[theorem]{Corollary}

\newtheorem{remark}{Remark}

\newtheorem{con}{Condition}[section]

\def\beq{\begin{equation}}
\def\nneq{\end{equation}}

\def\bthm{\begin{thm}}
\def\nthm{\end{thm}}

\def\blem{\begin{lem}}
\def\nlem{\end{lem}}
\def\bprf{\begin{proof}}

\def\bprop{\begin{prop}}
\def\nprop{\end{prop}}
\def\brmk{\begin{rem}}
\def\nrmk{\end{rem}}

\def\bexa{\begin{exa}}
\def\nexa{\end{exa}}
\def\bcor{\begin{cor}}
\def\ncor{\end{cor}}

\def\e{\varepsilon}

{\theorembodyfont{\rmfamily}
}

\title{Large deviations for stochastic models of two-dimensional second grade fluids driven by L\'evy noise
}

\footnotesize{\author{Jianliang Zhai $^{1,}$\thanks{zhaijl@ustc.edu.cn},\ \
 Tusheng Zhang$^{2,}$\thanks{Tusheng.Zhang@manchester.ac.uk},\ \ Wuting Zheng$^{1,}$\thanks{zwtzjr@mail.ustc.edu.cn}\\
 {\em $^1$ School of Mathematical Sciences,}\\
 {\em University of Science and Technology of China,}\\
 {\em Hefei, 230026, China}\\
 {\em $^2$ School of Mathematics, University of Manchester,}\\
 {\em Oxford Road, Manchester, M13 9PL, UK}\\
}
\date{}
\newenvironment{proof}{\par\noindent{\bf Proof:}}{\hspace*{\fill}$\blacksquare$\par}
\begin{document}
\maketitle

\noindent \textbf{Abstract:}
In this paper, we establish a large deviation principle for stochastic models of two-dimensional second grade fluids driven by L\'evy noise. The weak convergence method introduced by Budhiraja, Dupuis and Maroulas in
 \cite{BDM} plays a key role.

\vspace{4mm}

\numberwithin{equation}{section}
\section{Introduction}

The second grade fluids is an admissible
model of slow flow fluids,  which contains industrial fluids, slurries, polymer melts, etc.. It has attracted much attention
from a theoretical point of view, since it has properties of boundedness, stability and exponential decay, and has interesting connections with many other fluid models, see e.g. \cite{2003-Busuioc-p1119-1119}, \cite{1974-Dunn-p191-252}, \cite{1979-Fosdick-p145-152}, \cite{2001-Shkoller-p539-543} and references therein.
\vskip 0.2cm

Recently, taking into account random enviroment, the stochastic models of two-dimensional second grade fluids have been studied. For the case of
Wiener noises, we refer to \cite{CC,RS-10-01,RS-10,RS-12,WZZ,ZZ2}, where the authors obtained the existence and uniqueness of solutions,
 the behavior of the solutions as $\alpha\rightarrow0$,  Freidlin-Wentzell's large deviation principles and exponential mixing for the solutions. In the case
 of L\'evy noises, the global existence of a martingale solution was obtained in \cite{HRS}, and the existence and uniqueness of strong probabilistic solutions is studied in \cite{SZZ}.

\vskip 0.2cm
Based on the results in \cite{SZZ}, in this paper, we are concerned with Freidlin-Wentzell's large deviation principle of stochastic models for the incompressible second grade
fluids driven by L\'evy noises, which are given as follows:

\begin{align}\label{1.a}
\left\{
\begin{aligned}
& d(X^\e(t)-\alpha \Delta X^\e(t))+ \Big(-\kappa \Delta X^\e(t)+{\rm curl}(X^\e(t)-\alpha \Delta X^\e(t))\times X^\e(t)+\nabla\mathfrak{P}\Big)\,dt \\
& = F(X^\e(t),t)\,dt+\sqrt{\e}G(X^\e(t),t)\,dW(t) +\e\int_{\mathbb{Z}} \sigma(t,X^\e(t-),z)\widetilde{N}^{\e^{-1}}(dzdt),\quad \rm{ in }\ \mathcal{O}\times(0,T], \\
&\begin{aligned}
& {\rm{div}}\,X^\e=0 \quad &&\rm{in}\ \mathcal{O}\times(0,T]; \\
& X^\e=0  &&\rm{in}\ \partial \mathcal{O}\times[0,T]; \\
& X^\e(0)=X_0  &&\rm{in}\ \mathcal{O},
&\end{aligned}
\end{aligned}
\right.
\end{align}

\noindent where $\mathcal{O}$ is an open domain of $\mathbb{R}^2$; $X^\e=(X^\e_1,X^\e_2)$ and $\mathfrak{P}$ represent the random velocity and modified pressure respectively. $\mathbb{Z}$ is
 a locally compact Polish space. On a specified complete filtered probability space $(\Omega,\mathcal{F},\{\mathcal{F}_t\}_{t\in[0,T]},P)$, $W$ is a one-dimensional standard Brownian motion, and
$\widetilde{N}^{\epsilon^{-1}}$ is a compensated Poisson random measure on $[0,T]\times\mathbb{Z}$ with a
$\sigma$-finite
mean measure $\epsilon^{-1}\lambda_T\otimes \nu$, $\lambda_T$ is the Lebesgue measure on $[0,T]$ and $\nu$ is a $\sigma$-finite measure on $\mathbb{Z}$. The details of $(\Omega,\mathcal{F},\{\mathcal{F}_t\}_{t\in[0,T]},P, W, \widetilde{N}^{\epsilon^{-1}})$ will be given in Section 2.
\vskip 0.2cm

Due to the appearance of jumps in our setting, the Freidlin-Wentzell's large deviations are distinctively different to the Wiener case in \cite{ZZ2}.
We will apply the weak convergence approach introduced in \cite{BDM} and \cite{BCD} for the case of Poisson random measures.
This approach is mainly based on a variational representation formula for continuous time processes, and has been proved to be
 a powerful tool to establish the Freidlin-Wentzell large deviation principle for various
finite and infinite dimensional stochastic dynamical systems with irregular coefficients
driven by a non-Gaussian L\'evy noise, see for example \cite{Bao-Yuan}, \cite{BCD}, \cite{YZZ}, \cite{ZZ1}, \cite{Dong-Xiong-Zhai-Zhang}, \cite{Xiong-Zhai}. Because of the nature of the second grade fluids, technical difficulties arise even in the deterministic case. In addition to the complex structure of the system (\ref{1.a}), the nature of the nonlinear term ${\rm curl}\big(X^\e(t)-\alpha \Delta X^\e(t)\big)\times X^\e(t)dt$ implies that the solution $X^\e(t)$ should be in $H^3(\mathcal{O})\times H^3(\mathcal{O})$. The chain rule or $\rm{It\hat{o}'s}$ Formula of (\ref{1.a}) shows that the linear term $-\kappa \Delta X^\e(t)dt$ is not acting as a smoothing term like many nonlinear evolutions such as the Navier-Stokes equations. On the other hand, when applying the weak convergence method to the system (\ref{1.a}), it will be proved that the solutions of the controlled stochastic evolution equations (\ref{LDP Girsanov eq}), denoted by $\widetilde{X}^\e,\e>0$, have the following priori estimate: 
$$\sup_{0<\e<\e_0}E\Big[\sup_{t\in[0,T]}\|\widetilde{X}^\e(t)\|_{H^3(\mathcal{O})\times H^3(\mathcal{O})}^{2}\Big]<\infty.$$
This is non-trival, see Lemma \ref{LDP Lemma Estimates 1}.

\vskip 0.2cm

%
%

We organize this paper as follows. In Section 2, we introduce some functional spaces and $(\Omega,\mathcal{F},\{\mathcal{F}_t\}_{t\in[0,T]},\\
P, W, \widetilde{N}^{\epsilon^{-1}})$ appeared in the system (\ref{1.a}). In Section 3, we formulate the hypotheses. In section 4, we establish the large deviation principle.

\section{Notations and Preliminaries}

In this section, we first introduce some functional spaces and preliminaries that are needed in the paper, and then specify $(\Omega,\mathcal{F},\{\mathcal{F}_t\}_{t\in[0,T]},P, W, \widetilde{N}^{\epsilon^{-1}})$ in the system (\ref{1.a}).

\vskip 0.2cm

In this paper, we assume that $\mathcal{O}$ is a simply connected and bounded open domain of $\mathbb{R}^2$ with boundary $\partial \mathcal{O}$ of class $\mathcal{C}^{3,1}$. For $p\geq 1$ and $k\in\mathbb{N}$, we denote by $L^p(\mathcal{O})$
and $W^{k,2}(\mathcal{O})$ the usual $L^p$ and Sobolev spaces over $\mathcal{O}$ respectively.
Let $W^{k,2}_0(\mathcal{O})$ be
the closure in $W^{k,2}(\mathcal{O})$ of $\mathcal{C}^\infty_c(\mathcal{O})$ the space of infinitely differentiable functions with compact supports in $\mathcal{O}$. For simplicity, we write $H^k(\mathcal{O}):=W^{k,2}(\mathcal{O})$ and $H_0^k(\mathcal{O}):=W^{k,2}_0(\mathcal{O})$. We equip $H^1_0(\mathcal{O})$
with the scalar product
\begin{align*}
((u,v))=\int_\mathcal{O}\nabla u\cdot\nabla vdx=\sum_{i=1}^2\int_\mathcal{O}\frac{\partial u}{\partial x_i}\frac{\partial v}{\partial x_i}dx,
\end{align*}

\noindent where $\nabla$ is the gradient operator. It is well known that the norm $\|\cdot\|$ generated by this scalar product is equivalent to the usual norm of $H^1(\mathcal{O})$.

Throughout this paper, we set $\mathbb{Y}=Y\times Y$ for any Banach space $Y$.
Set
\begin{align*}
&\mathcal{C}=\Big\{u\in[\mathcal{C}^\infty_c(\mathcal{O})]^2 : {\rm div}\ u=0\Big\},\nonumber\\
&\mathbb{V}={\rm\ the\ closure\ of}\ \mathcal{C} {\rm\ in}\ \mathbb{H}^1(\mathcal{O}),\\
&\mathbb{H}={\rm\ the\ closure\ of}\ \mathcal{C}\ {\rm in}\ \mathbb{L}^2(\mathcal{O}).\nonumber
\end{align*}

We denote by $(\cdot,\cdot)$ and $|\cdot|$ the inner product in $\mathbb{L}^2(\mathcal{O})$(in $\mathbb{H}$) and the induced norm, respectively. The inner product and the norm of $\mathbb{H}^1_0(\mathcal{O})$
are denoted respectively by $((\cdot,\cdot))$ and $\|\cdot\|$. We endow the space $\mathbb{V}$ with the norm generated by the following inner product
\[
(u,v)_\mathbb{V}:=(u,v)+\alpha ((u,v)),\quad \text{for any } u,v\in\mathbb{V},
\]

\noindent and the norm in $\mathbb{V}$ is denoted by $\|\cdot\|_{\mathbb{V}}$. The $\rm Poincar\acute{e}$'s inequality implies that there exists a constant $\mathcal{P}>0$ such that the following inequalities holds
\begin{align}\label{Poincare}
(\mathcal{P}^2+\alpha)^{-1}\|v\|^2_\mathbb{V} \leq \|v\|^2
\leq\alpha^{-1}\|v\|^2_\mathbb{V},\quad \text{for any } v\in\mathbb{V}.
\end{align}

\vskip 0.2cm
We also introduce the following space
\[
\mathbb{W}=\big\{u\in\mathbb{V}: {\rm curl}(u-\alpha\Delta u)\in L^2(\mathcal{O})\big\},
\]
and endow it with the norm generated by the scalar product
\begin{align}\label{W}
(u,v)_\mathbb{W}:=\big({\rm curl}(u-\alpha\Delta u),{\rm curl}(v-\alpha\Delta v)\big).
\end{align}
The norm in $\mathbb{W}$ is denoted by $\|\cdot\|_{\mathbb{W}}$. It has been proved that, see e.g. \cite{CG,CO}, the following (algebraic and topological) identity holds:
\begin{align*}
\mathbb{W}=\big\{v\in\mathbb{H}^3(\mathcal{O}): {\rm div}\,v=0\ {\rm and}\  v|_{\partial \mathcal{O}}=0\big\},
\end{align*}
moreover, there exists a constant $C> 0$ such that
\begin{align}\label{W-02}
    |v|_{\mathbb{H}^3(\mathcal{O})}
\leq
    C\|v\|_\mathbb{W},\ \ \ \forall v\in \mathbb{W}.
\end{align}
This result states that the norm $\|\cdot\|_\mathbb{W}$ is equivalent to the usual norm in $\mathbb{H}^3(\mathcal{O})$.

Identifying the Hilbert space $\mathbb{V}$ with its dual space $\mathbb{V}^*$ by the Riesz representation, we get a
Gelfand triple
\begin{align*}
\mathbb{W}\subset \mathbb{V}\subset\mathbb{W}^*.
\end{align*}

\noindent We denote by $\langle f,v\rangle$ the dual relation between $f\in\mathbb{W}^*$ and $v\in\mathbb{W}$ from now on. It is easy to see
\[(v,w)_\mathbb{V}=\langle v,w\rangle,\ \ \ \forall\,v\in\mathbb{V},\ \ \forall\,w\in\mathbb{W}.\]

Note that the injection of $\mathbb{W}$ into $\mathbb{V}$ is compact,
thus there exists a sequence $\{e_i\}$ of elements of $\mathbb{W}$ which forms an orthonormal basis in $\mathbb{W}$, and an orthogonal system in $\mathbb{V}$, moreover this sequence verifies:
\begin{align}\label{Basis}
(v,e_i)_{\mathbb{W}}=\lambda_i(v,e_i)_{\mathbb{V}},\ \text{for any }v\in\mathbb{W},
\end{align}
where $0<\lambda_i\uparrow\infty$. From Lemma 4.1 in \cite{CG} we have
\begin{align}\label{regularity of basis}
e_i\in \mathbb{H}^4(\mathcal{O}),\ \ \forall\,i\in\mathbb{N}.
\end{align}

Consider the following ``generalized Stokes equations'':
\begin{align}\label{General Stokes}
\begin{aligned}
v-\alpha \Delta v &=f\quad {\rm in}\quad\mathcal{O},\\
{\rm div}\,v &=0\quad {\rm in}\quad\mathcal{O},\\
v &=0\quad {\rm on}\quad\partial \mathcal{O}.
\end{aligned}
\end{align}

\noindent The following result can be derived from \cite{Solonnikov} and also can be found in \cite{RS-10,RS-12}.
\begin{lem}\label{Lem GS}
Set $l=1,2,3$. Let $f$ be a function in $\mathbb{H}^l$, then the system (\ref{General Stokes}) has a unique solution $v$. Moreover if $f$ is an element of $\mathbb{H}^l\cap\mathbb{V}$, then $v\in \mathbb{H}^{l+2}\cap\mathbb{V}$, and the following relations hold
\begin{gather}
(v,g)_\mathbb{V}=(f,g),\quad \forall\, g\in \mathbb{V},\label{Eq GS-01}\\
|v|_{\mathbb{H}^{l+2}}\leq C|f|_{\mathbb{H}^l}.\label{Eq GS-02}
\end{gather}
\end{lem}

\vskip 0.3cm
Define the Stokes operator by
\begin{align}\label{Eq Stoke}
Au:=-\mathbb{P}\Delta u,\quad\forall\,u\in D(A)=\mathbb{H}^2(\mathcal{O})\cap\mathbb{V},
\end{align}
here the mapping $\mathbb{P}:\mathbb{L}^2(\mathcal{O})\longrightarrow\mathbb{H}$ is the usual Helmholtz-Leray projection. It follows from Lemma \ref{Lem GS} that the operator $(I+\alpha A)^{-1}$ defines an isomorphism from $\mathbb{H}^l(\mathcal{O})\cap\mathbb{H}$
into $\mathbb{H}^{l+2}(\mathcal{O})\cap\mathbb{V}$ for $l=1,2,3$. Moreover, for any $f\in\mathbb{H}^l(\mathcal{O})\cap\mathbb{V}$ and $g\in\mathbb{V}$, the following properties hold
\begin{align}\label{inverse operator transform}
((I+\alpha A)^{-1}f,g)_\mathbb{V}=(f,g),\nonumber\\
\|(I+\alpha A)^{-1}f\|_\mathbb{W}\leq C_A \|f\|_{\mathbb{V}}.
\end{align}

Let $\widehat{A}:=(I+\alpha A)^{-1}A$ , then $\widehat{A}$ is a
continuous linear operator from $\mathbb{H}^l(\mathcal{O})\cap\mathbb{V}$ onto itself for $l=2,3$, and satisfies
\begin{align}\label{A transform 01}
(\widehat{A}u,v)_\mathbb{V}=(Au,v)=((u,v)), \quad \forall\,u\in\mathbb{W},\ v\in\mathbb{V},
\end{align}
hence
\begin{align}\label{A transform 02}
(\widehat{A}u,u)_\mathbb{V}=\|u\|, \quad\forall\,u\in\mathbb{W}.
\end{align}

\noindent We also have
\begin{align}\label{A bound}
\|\widehat{A}u\|_{\mathbb{W}^{*}}\leq C\|\widehat{A}u\|_{\mathbb{V}}\leq C\|u\|_{\mathbb{V}}.
\end{align}

We recall the following estimates which can be found in \cite{RS-12}.

\begin{lem}\label{Lem B}
For any $u,v,w\in\mathbb{W}$, we have
\begin{align}\label{Ineq B 01}
    |({\rm curl}(u-\alpha\Delta u)\times v,w)|
\leq
    C\|u\|_{\mathbb{W}}\|v\|_\mathbb{V}\|w\|_{\mathbb{W}},
\end{align}

\noindent and
\begin{align}\label{Ineq B 02}
    |({\rm curl}(u-\alpha\Delta u)\times u,w)|
\leq
    C\|u\|^2_\mathbb{V}\|w\|_{\mathbb{W}}.
\end{align}
\end{lem}

Defining the bilinear operator $\widehat{B}(\cdot,\cdot):\ \mathbb{W}\times\mathbb{V}\longrightarrow\mathbb{W}^*$ by
\begin{align*}
\widehat{B}(u,v):=(I+\alpha A)^{-1}\mathbb{P}\big({\rm curl}(u-\alpha \Delta u)\times v\big).
\end{align*}
We have the following consequence from this lemma.

\begin{lem}\label{Lem-B-01}

For any $u\in\mathbb{W}$ and $v\in\mathbb{V}$, it holds that
\begin{align}\label{Eq B-01}
    \|\widehat{B}(u,v)\|_{\mathbb{W}^*}
\leq
    C\|u\|_\mathbb{W}|v|_\mathbb{V},
\end{align}

\noindent and
\begin{align}\label{Eq B-02}
\|\widehat{B}(u,u)\|_{\mathbb{W}^*}
\leq
    C_B\|u\|^2_\mathbb{V}.
\end{align}

\noindent In addition
\begin{align}\label{Eq B-03}
 \langle\widehat{B}(u,v),v\rangle=0, \quad\forall\,u, v\in\mathbb{W},
\end{align}

\noindent which implies
\begin{align}\label{Eq B-04}
 \langle\widehat{B}(u,v),w\rangle=-\langle\widehat{B}(u,w),v\rangle,\quad\forall\,u, v, w\in\mathbb{W}.
\end{align}

\end{lem}

\vskip 0.2cm

We are now to introduce $(\Omega,\mathcal{F},\mathbb{F}:=\{\mathcal{F}_t\}_{t\in[0,T]},P, W, \widetilde{N}^{\epsilon^{-1}})$.
\vskip 0.2cm

 Set $S$ be a locally compact Polish space. We put $M_{FC}(S)$ be the space of all Borel measures $\vartheta$ on $S$ such that  $\vartheta(K)<\infty$
for each compact set $K\subseteq S$. Endow $M_{FC}(S)$ with the weakest topology, denoted it by $\mathcal{T}(M_{FC}(S))$, such that for each $f\in C_c(S)$ the mapping
$\vartheta\in M_{FC}(S)\rightarrow \int_Sf(s)\vartheta(ds)$ is continuous. This topology is metrizable such that $M_{FC}(S)$ is a Polish space, see \cite{BDM} for more details.
\vskip 0.2cm

Recall that $\mathbb{Z}$ is a locally compact Polish space, and in this paper, we assume that $\nu$ is a given element of $M_{FC}(\mathbb{Z})$.  We specify the underlying probability space $(\Omega, \mathcal{F}, {\mathbb{F}}:=\{\mathcal{F}_t\}_{t\in [0,T]},P)$ in the following way:
\begin{align*}
  \Omega:=C\big([0,T],\mathbb{R}\big)\times M_{FC}\big([0,T]\times \mathbb{Z}
  \times [0,\infty)\big),\qquad \mathcal{F}:=\mathcal{B}(C\big([0,T],\mathbb{R}\big))\otimes\mathcal{T}(M_{FC}\big([0,T]\times \mathbb{Z}
  \times [0,\infty)).
\end{align*}
We introduce the functions
\begin{align*}
&W\colon \Omega \rightarrow C\big([0,T],\mathbb{R}\big),
\qquad W(\alpha,\beta)(t)=\alpha(t),\\
& N\colon \Omega \rightarrow M_{FC}\big([0,T]\times \mathbb{Z}\times [0,\infty)\big),\qquad  N(\alpha,\beta)=\beta.
\end{align*}
Define for each $t\in [0,T]$ the $\sigma$-algebra
\begin{align*}
\mathcal{G}_{t}:=\sigma\left(\left\{\big(W(s), \, N((0,s]\times A)\big):\,
0\leq s\leq t,\,A\in \mathcal{B}\big(\mathbb{Z}\times [0,\infty)\big)\right\}\right).
\end{align*}
Let $\lambda_T$ and $\lambda_\infty$ be Lebesgue measure on $[0,T]$ and $[0,\infty)$ respectively. It follows from \cite[Sec.I.8]{Ikeda-Watanabe} that there exists a unique probability measure $P$
 on $(\Omega,\mathcal{F})$ such that:
\begin{enumerate}
\item[(a)] $W$ is one-dimension standard Brownian motion;
\item[(b)] $N$ is a Poisson random measure on $\Omega$ with intensity measure $\lambda_T\otimes\nu\otimes \lambda_\infty$;
\item[(c)] $W$ and $N$ are independent.
\end{enumerate}
We denote by $\mathbb{F}:=\{{\mathcal{F}}_{t}\}_{t\in[0,T]}$ the
$P$-completion of $\{\mathcal{G}_{t}\}_{t\in[0,T]}$ and by
$\mathcal P$ the $\mathbb{F}$-predictable $\sigma$-field
on $[0,T]\times \Omega$.
Define
\begin{align*}
{\mathcal{A}}
:=\left\{\varphi\colon [0,T]\times {\mathbb{Z}}\times\Omega\to [0,\infty):
\, (\mathcal{P}\otimes\mathcal{B}({\mathbb{Z}}))\setminus\mathcal{B}[0,\infty)\text{-measurable}\right\}.
\end{align*}

For $\varphi\in{\mathcal{A}}$, define a
counting process $N^{\varphi}$ on $[0,T]\times {{\mathbb{Z}}}$ by
   \begin{align*}
      N^\varphi((0,t]\times A)=\int_{(0,t]\times A\times (0,\infty)}1_{[0,\varphi(s,z)]}(r)\, N(ds, dz, dr),
   \end{align*}
for $t\in[0,T]$ and $A\in\mathcal{B}({\mathbb{Z}})$. When $\varphi(s,z,\omega)=\epsilon^{-1}$, we write $N^\varphi=N^{\epsilon^{-1}}$.
It is easy to see that $N^{\epsilon^{-1}}$ is a Poisson random measure on $[0,T]\times\mathbb{Z}$ with a mean measure $\epsilon^{-1}\lambda_T\otimes\nu$. We denote $\widetilde{N}^{\epsilon^{-1}}$ the compensated Poisson random measure respect to $N^{\epsilon^{-1}}$.

\vskip 0.2cm
At the end of this section, we introduce the following notions which will be used later.
\vskip 0.2cm

For each $f\in L^2([0,T],\mathbb{R})$, we introduce the quantity
\begin{align*}
Q_{1}(f)
:=\frac{1}{2}\int_{0}^{T}|{f(s)}|^{2}\,ds,
\end{align*}
and we define for each $m\in\mathbb{N}$ the space
\begin{align*}
 S_1^m:=\Big\{f\in L^{2}([0,T],\mathbb{R}):\,Q_1(f)\leq m\Big\}.
\end{align*}
Equiped with the weak topology, $S_1^m$ is a compact subset of $L^2([0,T],\mathbb{R}).$
We will throughout consider $S_1^m$ endowed with this topology.
By defining the function
\begin{align*}
\ell:[0,\infty)\rightarrow[0,\infty), \qquad
\ell(x)=x\log x-x +1,
\end{align*}
we introduce for each measurable function
$g\colon [0,T]\times {\mathbb{Z}}\to [0,\infty)$ the quantity
\begin{align*}
Q_2(g):=\int_{[0,T]\times {\mathbb{Z}}}\ell\big(g(s,z)\big) \,ds \,\nu(dz).
\end{align*}
Define for each $m\in\mathbb{N}$ the space
\begin{align*}
     S_2^m:=\Big\{g:[0,T]\times {\mathbb{Z}}\rightarrow[0,\infty):\,Q_2(g)\leq m\Big\}.
\end{align*}
A function $g\in S_2^m$ can be identified with a measure $\nu_T^g\in M_{FC}([0,T]\times {\mathbb{Z}})$, defined by
   \begin{align}\label{eq.corres-func-meas}
      \nu_T^g(A)=\int_A g(s,z)\,ds\,\nu(dz)\ \quad\text{ for all } A\in\mathcal{B}([0,T]\times {\mathbb{Z}}).
   \end{align}
This identification induces a topology on $S_2^m$ under which $S_2^m$ is a compact space, see \cite{BCD}.
Throughout, we use this topology on $S_2^m$.\\
\indent Define $S^m=S_1^m\times S_2^m$ and $\mathcal{S}=\bigcup_{m\geq1}S^m$.

\section{Hypotheses}

In this section, we will formulate precise assumptions on coefficients.

Let
\begin{align*}
F & :\mathbb{V}\times[0,T]\longrightarrow\mathbb{V}; \\
G & :\mathbb{V}\times[0,T]\longrightarrow\mathbb{V}; \\
\sigma & :[0,T]\times \mathbb{V}\times \mathbb{Z}\longrightarrow\mathbb{V},
\end{align*}
be given measurable maps. We introduce the following notations:
\begin{gather*}
\widehat{F}(u,t):=(I+\alpha A)^{-1} F(u,t) ;\\
\widehat{G}(u,t):=(I+\alpha A)^{-1} G(u,t) ;\\
\widehat{\sigma}(t,u,z):=(I+\alpha A)^{-1} \sigma(t,u,z) .
\end{gather*}

The following assumptions are from \cite{SZZ}, which guarantee that (\ref{1.a}) admits a unique solution.
\begin{con}\label{condition 1}
There exists constants $C_F,C_G,C\geq 0$, such that the following conditions hold for all $u_1,u_2,u\in\mathbb{V}$ and $t\in [0,T]$:

\noindent (1) (Lipschitz)
\begin{eqnarray}\label{Condition F}
\|F(u_1,t)-F(u_2,t)\|_{\mathbb{V}}^2\leq C_F\|u_1-u_2\|_{\mathbb{V}}^2,\ \ F(0,t)=0,
\end{eqnarray}
\begin{eqnarray}\label{Condition G}
\|G(u_1,t)-G(u_2,t)\|_{\mathbb{V}}^2\leq C_G\|u_1-u_2\|_{\mathbb{V}}^2,\ \ G(0,t)=0,
\end{eqnarray}
\begin{eqnarray}\label{Condition S1}
\int_\mathbb{Z} \|\sigma(t,u_1,v)-\sigma(t,u_2,z)\|_{\mathbb{V}}^2\,\nu(dz)\leq C_L\|u_1-u_2\|_{\mathbb{V}}^2.
\end{eqnarray}

\noindent (2) (Growth)
\begin{eqnarray}\label{Condition S2}
\int_\mathbb{Z} \|\sigma(t,u,z)\|_{\mathbb{V}}^{2q}\,\nu(dz)\leq C(1+\|u\|_{\mathbb{V}}^{2q}),\ \ \ q=1,2.
\end{eqnarray}
\end{con}
\vskip0.2cm

Let
\begin{eqnarray*}
\big\|\sigma(t,z)\big\|_{0,\mathbb{V}}=\sup_{u\in\mathbb{V}}\frac{\|\sigma(t,u,z)\|_\mathbb{V}}{1+\|u\|_\mathbb{V}},\ \ \ (t,z)\in[0,T]\times\mathbb{Z},
\end{eqnarray*}
\begin{eqnarray*}
\big\|\sigma(t,z)\big\|_{1,\mathbb{V}}=\sup_{u_1,u_2\in\mathbb{V},u_1\neq u_2}\frac{\|\sigma(t,u_1,z)-\sigma(t,u_2,z)\|_\mathbb{V}}{\|u_1-u_2\|_\mathbb{V}},\ \ \ (t,z)\in[0,T]\times\mathbb{Z}.
\end{eqnarray*}

To study large deviation principle of (\ref{1.a}), besides Condition \ref{condition 1}, we further need
\begin{con}\label{condition 2}
\noindent (1)($L^2$-integrability) For $i=0,1$, $\|\sigma(\cdot,\cdot)\|_{i,\mathbb{V}}\in L^2(\lambda_T\otimes\nu)$, i.e.
$$\int_{[0,T]}\int_{{\mathbb{Z}}}\|\sigma(t,z)\|_{i,\mathbb{V}}^2\nu(dz)dt<\infty.$$
\noindent (2)(Exponential integrability) For $i=0,1$, there exists $\delta_1^i>0$ such that for all $E\in{\mathcal{B}}([0,T]\times\mathbb{Z})$ satisfying $\lambda_T\otimes\nu(E)<\infty$, the following holds
$$\int_E e^{\delta_1^i\|\sigma(t,z)\|_{i,\mathbb{V}}^2}\nu(dz)dt<\infty.$$
\end{con}

\begin{remark}\label{remark}
Condition \ref{condition 2} implies that, for every $\delta_2^i>0$ and for all $E\in{\mathcal{B}}([0,T]\times\mathbb{Z})$ satisfying $\lambda_T\otimes\nu(E)<\infty$,
$$\int_E e^{\delta_2^i\|\sigma(t,z)\|_{i,\mathbb{V}}}\nu(dz)dt<\infty.$$
\end{remark}
The following lemma was proved in Budhiraja, Chen and Dupuis \cite{BCD}. For the second part of this lemma, the case $i=0$ can be found in Remark 2 of Yang, Zhai and Zhang \cite{YZZ}, and the case $i=1$ can be proved similarly. We omit its proof.
\begin{lem}\label{H-lemma 1}
Under Condition \ref{condition 1} and \ref{condition 2},\\
\noindent (1) For $i=0,1$ and every $m\in\mathbb{N}$,
\begin{eqnarray}\label{H-1}
C_{i,2}^{m}:=\sup_{g\in{S_2^m}}\int_0^T\int_{\mathbb{Z}}\|\sigma(s,z)\|_{i,\mathbb{V}}^{2}(g(s,z)+1)\nu(dz)ds<\infty,
\end{eqnarray}
\begin{eqnarray}\label{H-2}
C_{i,1}^m:=\sup_{g\in{S_2^m}}\int_0^T\int_{\mathbb{Z}}\|\sigma(s,z)\|_{i,\mathbb{V}}|g(s,z)-1|\nu(dz)ds<\infty.
\end{eqnarray}
\noindent (2) For every $\eta>0$, there exist $\delta>0$ such that for any $A\subset[0,T]$ satisfying $\lambda_T(A)<\delta$
\begin{eqnarray}\label{H-3}
\sup_{g\in{S_2^m}}\int_{A}\int_{\mathbb{Z}}\|\sigma(s,z)\|_{i,\mathbb{V}}|g(s,z)-1|\nu(dz)ds\leq\eta.
\end{eqnarray}
\end{lem}

We also need the following lemma, the proof of which can be found in Budhiraja, Chen and Dupuis \cite{BCD}.
\begin{lem}\label{H-lemma 3}
Fix $m\in\mathbb{N}$, and let $g_n,g\in S_2^m$ be such that $g_n\rightarrow g$ as $n\rightarrow\infty$. Let $h:[0,T]\times\mathbb{Z}\rightarrow\mathbb{R}$ be a measurable function such that
$$\int_{{[0,T]\times\mathbb{Z}}}|h(s,z)|^2\nu(dz)ds<\infty,$$
and for all $\delta\in(0,\infty)$
$$\int_E\text{exp}(\delta|h(s,z)|)\nu(dz)ds<\infty.$$
for all $E\in{\mathcal{B}}([0,T]\times\mathbb{Z})$ satisfying $\lambda_T\otimes\nu(E)<\infty$. Then
$$\lim_{n\rightarrow\infty}\int_{{[0,T]\times\mathbb{Z}}}h(s,z)(g_n(s,z)-1)\nu(dz)ds= \int_{{[0,T]\times\mathbb{Z}}}h(s,z)(g(s,z)-1)\nu(dz)ds.$$
\end{lem}

Let $\mathbb{K}$ be a separable Hilbert space. Given $p>1$, $\beta\in(0,1)$, let $W^{\beta,p}([0,T],\mathbb{K})$ be the
 space of all $u\in L^p([0,T],\mathbb{K})$ such that
$$
\int_0^T\int_0^T\frac{\|u(t)-u(s)\|^p_\mathbb{K}}{|t-s|^{1+\beta p}}dtds<\infty,
$$
endowed with the norm
$$
\|u\|^p_{W^{\beta,p}([0,T],\mathbb{K})}:=\int_0^T\|u(t)\|^p_{\mathbb{K}}dt+\int_0^T\int_0^T\frac{\|u(t)-u(s)\|^p_\mathbb{K}}{|t-s|^{1+\beta p}}dtds.
$$

The following result is a variant of the criteria for compactness proved in Lions \cite{Lions} (Sect. 5, Ch. I)
 and Temam \cite{Temam 1983} (Sect. 13.3).
\begin{lem}\label{H-lemma 4}{\rm
Let $\mathbb{K}_0\subset \mathbb{K}\subset \mathbb{K}_1$ be Banach spaces, $\mathbb{K}_0$ and $\mathbb{K}_1$ reflexive, with compact embedding of $\mathbb{K}_0$ into $\mathbb{K}$.
For $p\in(1,\infty)$ and $\beta\in(0,1)$, let $\Lambda$ be the space
$$
\Lambda=L^p([0,T],\mathbb{K}_0)\cap W^{\beta,p}([0,T],\mathbb{K}_1)
$$
endowed with the natural norm. Then the embedding of $\Lambda$ into $L^p([0,T],\mathbb{K})$ is compact.
}\end{lem}

\section{Large deviation}

\subsection{Skeleton equations}

\indent As a first step we show that, for every $q=(f,g)\in\mathcal{S}$ the deterministic integral equation
\begin{eqnarray}\label{skeleton equations}
\widetilde{X}^{q}(t)=X_0-\kappa\int_0^t\widehat{A}\widetilde{X}^{q}(s)ds-\int_0^t\widehat{B}(\widetilde{X}^{q}(s),\widetilde{X}^{q}(s))ds
    +\int_0^t\widehat{F}(\widetilde{X}^{q}(s),s)ds\nonumber\\
+\int_0^t\widehat{G}(\widetilde{X}^{q}(s),s)f(s)ds+\int_0^t\int_{\mathbb{Z}}\widehat{\sigma}(s,\widetilde{X}^{q}(s),z)\big(g(s,z)-1\big)\nu(dz)ds
\end{eqnarray}
has a unique solution. That is

\begin{thm}\label{Theorem skeletons}
Fix $X_0\in\mathbb{W}$ , $q=(f,g)\in\mathcal{S}$. Suppose Condition \ref{condition 1} and \ref{condition 2} hold. Then there exists a unique $\widetilde{X}^q\in C([0,T],\mathbb{V})$ satisfy (\ref{skeleton equations}).
Moreover, for any $m\in\mathbb{N}$, there exists a constant $C_m$ such that
\begin{eqnarray}\label{Estimation 1W}
\sup_{q\in S^m}\sup_{s\in[0,T]}\|\widetilde{X}^q(s)\|_\mathbb{W}^2\leq C_m.
\end{eqnarray}
\end{thm}
\begin{proof}
First, let $\Phi_n:\mathbb{R}\rightarrow[0,1]$ be a smooth function such that $\Phi_n(r)=1$ if $|r|\leq n$, $\Phi_n(r)=0$ if $|r|> n+1$. Set $\mathcal{X}_n(u)=\Phi_n(\|u\|_\mathbb{V}), u\in\mathbb{V}$. Let $\mathbb{P}_n$ be the projection operator from $\mathbb{V}$ to $\mathbb{V}$ defined as
\begin{eqnarray*}
\mathbb{P}_{n}u=\sum_{i=1}^n\big(u,e_i\big)_{\mathbb{W}}e_i,\ \ u\in\mathbb{V}.\\
\end{eqnarray*}
Set
\begin{eqnarray*}
\mathbb{P}_{n}\mathbb{V}:=\text{Span}\{e_1,\cdots,e_n\},
\end{eqnarray*}
$$\widehat{B}_n(u,u)=\mathcal{X}_n(u)\widehat{B}(u,u),\ \ u\in\mathbb{P}_n\mathbb{V}.$$
\indent (\ref{regularity of basis}) and Lemma \ref{Lem-B-01} implies that $\widehat{B}_n$ is a global Lipschitz operator from $\mathbb{P}_n\mathbb{V}$ into $\mathbb{P}_n\mathbb{V}$. Repeating the same arguments as in the proof of Theorem 4.1 in \cite{ZZ1}, there exists a unique $X_n\in C([0,T],\mathbb{P}_n\mathbb{V})$ satisfying the following auxiliary PDE:
\begin{eqnarray}\label{skeleton projections 1}
dX_n(t)&=&
-\kappa\mathbb{P}_n\widehat{A}X_n(t)dt-\mathbb{P}_n\widehat{B}_n\big(X_n(t),X_n(t)\big)dt+\mathbb{P}_n\widehat{F}(X_n(t),t)dt\nonumber\\
& &+\mathbb{P}_n\widehat{G}(X_n(t),t)f(t)dt+\mathbb{P}_n\int_{\mathbb{Z}}\widehat{\sigma}(t,X_n(t),z)\big(g(t,z)-1\big)\nu(dz)dt
\end{eqnarray}
with the initial value $X_n(0)=\mathbb{P}_{n}X_0\in\mathbb{W}$.\\
The next thing is to show that

\begin{eqnarray}\label{Estimation 2W}
\sup_{n}\sup_{s\in[0,T]}\|X_n(s)\|_\mathbb{W}^{2}\leq C<\infty,
\end{eqnarray}
and for $\alpha\in(0,\frac{1}{2})$
\begin{eqnarray}\label{Estimation 2}
\sup_{n}\|X_n\|_{\mathbb{W}^{\alpha,2}([0,T],\mathbb{W}^*)}\leq C_\alpha<\infty.
\end{eqnarray}

By (\ref{skeleton projections 1}), we have
\begin{eqnarray}\label{Estimation 2W eq 1}
d\big(X_n(t),e_i\big)_\mathbb{V}
&=&
-\kappa\big(\widehat{A}X_n(t),e_i\big)_\mathbb{V}dt
    -\big(\mathbb{P}_n\widehat{B}_n\big(X_n(t),X_n(t)\big),e_i\big)_{\mathbb{V}}dt+\big(\widehat{F}(X_n(t),t),e_i\big)_\mathbb{V}dt\nonumber\\
& &+(\widehat{G}(X_n(t),t)f(t),e_i)_\mathbb{V}dt
+\big(\int_{\mathbb{Z}}\widehat{\sigma}(t,X_n(t),z)\big(g(t,z)-1\big)\nu(dz),e_i\big)_\mathbb{V}dt,\nonumber\\
\end{eqnarray}
for $i=1,2,\cdots,n$.\\
\indent By (\ref{Basis}), multiplying both sides of the equation (\ref{Estimation 2W eq 1}) by $\lambda_i$, we get
\begin{eqnarray}\label{Estimation 2W eq 2}
d\big(X_n(t),e_i\big)_\mathbb{W}
&=&
-\kappa\big(\widehat{A}X_n(t),e_i\big)_\mathbb{W}dt
    -\big(\mathbb{P}_n\widehat{B}_n\big(X_n(t),X_n(t)\big),e_i\big)_{\mathbb{W}}dt+\big(\widehat{F}(X_n(t),t),e_i\big)_\mathbb{W}dt\nonumber\\
& &+\big(\widehat{G}(X_n(t),t)f(t),e_i\big)_\mathbb{W}dt
+\big(\int_{\mathbb{Z}}\widehat{\sigma}(t,X_n(t),z)\big(g(t,z)-1\big)\nu(dz),e_i\big)_\mathbb{W}dt,\nonumber
\end{eqnarray}
for $i=1,2,\cdots,n$.\\
\indent By a calculation of $(X_n(t),e_i)_\mathbb{W}^2$ and summing over $i$ from $1$ to $n$ yields
\begin{eqnarray}\label{Estimation 2W eq 3}
\|X_n(t)\|_\mathbb{W}^2
&=&
\|X_n(0)\|_\mathbb{W}^2-2\kappa\int_0^t(\widehat{A}X_n(s),X_n(s))_\mathbb{W}ds
    -2\int_0^t(\mathbb{P}_n\widehat{B}_n\big(X_n(s),X_n(s)\big),X_n(s))_{\mathbb{W}}ds\nonumber\\
& &
+2\int_0^t(\widehat{F}(X_n(s),s),X_n(s))_\mathbb{W}ds
+2\int_0^t(\widehat{G}(X_n(s),t)f(s),X_n(s))_\mathbb{W}ds\nonumber\\
& &
+2\int_0^t(\int_{\mathbb{Z}}\widehat{\sigma}(s,X_n(s),z)\big(g(s,z)-1\big)\nu(dz),X_n(s))_\mathbb{W}ds.
\end{eqnarray}
Noticing that (see (4.11) in \cite{SZZ}, (4.61) in \cite{RS-10})
\begin{eqnarray*}
\big(\mathbb{P}_n\widehat{B}_n\big(X_n(s),X_n(s)\big),X_n(s)\big)_{\mathbb{W}}=0,
\end{eqnarray*}
\begin{eqnarray}\label{curl}
\big|curl(v)|^2\leq\frac{2}{\alpha}\|v\|_\mathbb{V}^2 \text{ for any } v\in\mathbb{V},
\end{eqnarray}
and
\begin{eqnarray*}
(\widehat{A}X_n(s),X_n(s))_\mathbb{W}
=
\frac{1}{\alpha}\|X_n(s)\|_\mathbb{W}^2-\frac{1}{\alpha}\Big(curl\big(X_n(s)\big),curl\big(X_n(s)-\alpha\Delta X_n(s)\big)\Big),
\end{eqnarray*}
we have
\begin{eqnarray}\label{Estimation 2W eq 4}
& &\|X_n(t)\|_\mathbb{W}^2+\frac{2\kappa}{\alpha}\int_0^t\|X_n(s)\|_\mathbb{W}^2ds\nonumber\\
&=&
\|X_n(0)\|_\mathbb{W}^2+\frac{2\kappa}{\alpha}\int_0^t\Big(curl\big(X_n(s)\big),curl\big(X_n(s)-\alpha\Delta X_n(s)\big)\Big)ds\nonumber\\
& &
+2\int_0^t(\widehat{F}(X_n(s),s),X_n(s))_\mathbb{W}ds
+2\int_0^t(\widehat{G}(X_n(s),s)f(s),X_n(s))_\mathbb{W}ds\nonumber\\
& &
+2\int_0^t(\int_{\mathbb{Z}}\widehat{\sigma}(s,X_n(s),z)\big(g(s,z)-1\big)\nu(dz),X_n(s))_\mathbb{W}ds\\
&\leq&
\|X_0\|_\mathbb{W}^2+\frac{2\kappa}{\alpha}\int_0^t\|X_n(s)\|_\mathbb{W}\big|curl\big(X_n(s)\big)\big|ds\nonumber\\
& &
+2\int_0^t\|X_n(s)\|_\mathbb{W}\big|curl\big(F(X_n(s),s)\big)\big|ds\nonumber\\
& &
+2\int_0^t\|X_n(s)\|_\mathbb{W}\big|curl\big(G(X_n(s),s)f(s)\big)\big|ds\nonumber\\
& &
+2\int_0^t\|X_n(s)\|_\mathbb{W}\big|curl\big(\int_{\mathbb{Z}}{\sigma}(t,X_n(s),z)\big(g(s,z)-1\big)\nu(dz)\big)\big|ds\nonumber\\
&\leq&
\|X_0\|_\mathbb{W}^2+C\int_0^t\|X_n(s)\|_\mathbb{W}^2ds
+C\int_0^t\|X_n(s)\|_\mathbb{W}^2\big(1+|f(s)|^2\big)ds\nonumber\\
& &
+C\int_0^t\big(1+\|X_n(s)\|_\mathbb{W}^2\big)\int_{\mathbb{Z}}\|{\sigma}(s,,z)\|_{0,\mathbb{V}}|g(s,z)-1|\nu(dz)ds\nonumber\\
&\leq&
\|X_0\|_\mathbb{W}^2+C\int_0^t\int_{\mathbb{Z}}\|{\sigma}(s,,z)\|_{0,\mathbb{V}}|g(s,z)-1|\nu(dz)\big)ds\nonumber\\
& &
+C\int_0^t\|X_n(s)\|_\mathbb{W}^2\Big\{1+|f(s)|^2
+\int_{\mathbb{Z}}\|{\sigma}(s,,z)\|_{0,\mathbb{V}}|g(s,z)-1|\nu(dz)\Big\}ds,\nonumber
\end{eqnarray}
we have used Condition \ref{condition 1}.\\
\indent By Lemma \ref{H-lemma 1} and Gronwall's inequality, we get (\ref{Estimation 2W}).
\vskip 0.2cm
Now we proof (\ref{Estimation 2}). By (\ref{skeleton projections 1})
\begin{eqnarray*}
X_n(t)
&=&\mathbb{P}_n X_n(0)-\kappa\int_0^t\mathbb{P}_n\widehat{A}X_n(s)ds-\int_0^t\mathbb{P}_n\widehat{B}_n(X_n(s),X_n(s))ds\nonumber\\
& &+\int_0^t\mathbb{P}_n\widehat{F}(X_n(s),s)ds+\int_0^t\mathbb{P}_n\widehat{G}(X_n(s),s)f(s)ds\nonumber\\
& &+\int_0^t\int_\mathbb{Z}\mathbb{P}_n\widehat{\sigma}(s,X_n(s),z)(g(s,z)-1)\nu(dz)ds\nonumber\\
&:=&J_n^1+J_n^2(t)+J_n^3(t)+J_n^4(t)+J_n^5(t)+J_n^6(t).
\end{eqnarray*}
\indent Choose $\alpha\in(0,\frac{1}{2})$, using the similar arguments as the proof of Proposition 4.5 in Zhai and Zhang \cite{ZZ2}, we can obatin
\begin{eqnarray*}
& &\|J_n^1\|_{W^{\alpha,2}([0,T],\mathbb{V})}+\|J_n^2\|_{W^{\alpha,2}([0,T],{\mathbb{V}})}
+\|J_n^3\|_{W^{\alpha,2}([0,T],{\mathbb{W}}^\ast)}\nonumber\\
& &
+\|J_n^4\|_{W^{\alpha,2}([0,T],{\mathbb{V}})}+\|J_n^5\|_{W^{\alpha,2}([0,T],{\mathbb{V}})}\leq L_{\alpha}<\infty,
\end{eqnarray*}
here $L_{\alpha}$ depends on $\alpha$ and $C$ from (\ref{Estimation 2W}).\\
\indent Using the similar arguments as (4.20) in Zhai and Zhang \cite{ZZ1}, we can obtain
\begin{eqnarray*}
\|J_n^6\|_{W^{\alpha,2}([0,T],{\mathbb{W}}^\ast)}\leq L_{\alpha}<\infty.
\end{eqnarray*}
Combining above all inequalities, we obtain (\ref{Estimation 2}).
\vskip 0.2cm
\indent The estimates (\ref{Estimation 2W}) and (\ref{Estimation 2}) enable us to assert the existence of $X\in L^\infty([0,T],\mathbb{W})$ and a sub-sequence $X_{m'}$ such that, as $m'\rightarrow\infty$\\
\indent{1.}  $X_{m'}\rightarrow X$ weakly in $L^2([0,T],\mathbb{W})$,\\
\indent{2.}  $X_{m'}\rightarrow X$ in the weak-star topology of $L^\infty([0,T],\mathbb{W})$.\\
Lemma \ref{H-lemma 4} has been used to obtain claim 3:\\
\indent{3.}  $X_{m'}\rightarrow X$ strongly in $L^2([0,T],\mathbb{V})$.
\vskip 0.2cm
\indent Finally, we use the similar argument as in the proof of Theorem 4.1 in \cite{ZZ1}, we can conclude that $X$ is a solution of (\ref{skeleton equations}) and refer to Tenam \cite{Temam 1979} Chapter 3, $X\in C([0,T],\mathbb{V})$. (\ref{Estimation 2W}) also implies that
\begin{eqnarray*}
\sup_{q\in S^m}\sup_{t\in[0,T]}\|X(t)\|_\mathbb{W}^2\leq C_m<\infty.
\end{eqnarray*}
\indent (Uniqueness) Let us assume that $X$ and $Y$ are two solutions of (\ref{skeleton equations}), and let $Z=X-Y$. We have
\begin{eqnarray}
& &\|Z(t)\|_\mathbb{V}^2+2\kappa\int_0^t\|Z(s)\|^2ds\nonumber\\
&=&-2\int_0^t\big\langle\widehat{B}(X(s),X(s))-\widehat{B}(Y(s),Y(s)),X(s)-Y(s)\big\rangle_{{\mathbb{W}}^\ast,\mathbb{W}}ds\nonumber\\
& &+2\int_0^t\big(\widehat{F}(X(s),s)-\widehat{F}(Y(s),s),X(s)-Y(s)\big)_\mathbb{V}ds\nonumber\\
& &+2\int_0^t\big(\widehat{G}(X(s),s)f(s)-\widehat{G}((Y(s),s)f(s),X(s)-Y(s)\big)_\mathbb{V}ds\nonumber\\
& &+2\int_0^t\int_{\mathbb{Z}}\big(\widehat{\sigma}(s,X(s),z)(g(s,z)-1)-\widehat{\sigma}(s,Y(s),z)(g(s,z)-1),X(s)-Y(s)\big)_\mathbb{V}\nu(dz)ds\nonumber\\
&:=&I_1(t)+I_2(t)+I_3(t)+I_4(t).
\end{eqnarray}
By Lemma \ref{Lem-B-01}, we get
\begin{eqnarray}
|I_1(t)|
&\leq&
\int_0^t\big|\big\langle\widehat{B}(X(s)-Y(s),X(s)-Y(s)),X(s)\big\rangle_{{\mathbb{W}}^\ast,\mathbb{W}}\big|ds\nonumber\\
&\leq&
C_B\int_0^t\|Z(s)\|_\mathbb{V}^2\|X(s)\|_\mathbb{W}ds.
\end{eqnarray}
By Condition \ref{condition 1}, we get
\begin{eqnarray}
|I_2(t)|
&\leq&
2\int_0^t\|\widehat{F}(X(s),s)-\widehat{F}(Y(s),s)\|_\mathbb{V}\|X(s)-Y(s)\|_\mathbb{V}ds\nonumber\\
&\leq&
2C_F\int_0^t\|Z(s)\|_\mathbb{V}^2ds,
\end{eqnarray}
and
\begin{eqnarray}
|I_3(t)|
&\leq&
2\int_0^t\|\widehat{G}(X(s),s)-\widehat{G}(Y(s),s)\|_\mathbb{V}|f(s)|\|X(s)-Y(s)\|_\mathbb{V}ds\nonumber\\
&\leq&
2C_G\int_0^t\|Z(s)\|_\mathbb{V}^2(1+|f(s)|^2)ds.
\end{eqnarray}
For $I_4$, we have
\begin{eqnarray}
|I_4(t)|
&\leq&
2\int_0^t\int_\mathbb{Z}\|\widehat{\sigma}(s,X(s),z)-\widehat{\sigma}(s,Y(s),z)\|_\mathbb{V}|g(s,z)-1|\nu(dz)\|X(s)-Y(s)\|_\mathbb{V}ds\nonumber\\
&\leq&
2\int_0^t\int_\mathbb{Z}\|\sigma(s,z)\|_{1,\mathbb{V}}|g(s,z)-1|\nu(dz)\|Z(s)\|_\mathbb{V}^2ds.
\end{eqnarray}
Setting
\begin{eqnarray}
\varphi(s)=C_B\|X(s)\|_\mathbb{W}+2C_F+2C_G(1+|f(s)|^2)+2\int_\mathbb{Z}\|\sigma(s,z)\|_{1,\mathbb{V}}|g(s,z)-1|\nu(dz),
\end{eqnarray}
we have
\begin{eqnarray}
\|Z(t)\|_\mathbb{V}^2+2\kappa\int_0^t\|Z(s)\|^2ds\leq\int_0^t\varphi(s)\|Z(s)\|_\mathbb{V}^2ds.
\end{eqnarray}
By Lemma \ref{H-lemma 1} and Gronwall's equality, we can conclude $X=Y$.\\
The proof is complete.
\end{proof}

\subsection{The main result}

\indent We are now ready to state the main result. Recall that for $q=(f,g)\in\mathcal{S},\nu_T^g(dsdz)=g(s,z)\nu(dz)ds$, define
\begin{eqnarray}
\mathcal{G}^0\Big(\int_0^\cdot f(s)ds,\nu_T^g\Big)=\widetilde{X}^q\ \ \ \ \text{for }q=(f,g)\in\mathcal{S} \text{ as given in Theorem }\ref{Theorem skeletons}.
\end{eqnarray}
\indent The next theorem is contained in Theorem 3.2 in Shang, Zhai and Zhang \cite{SZZ}.
\begin{thm}\label{LDP Thm solution}
Assume Condition \ref{condition 1}, if $X_0\in\mathbb{W}$, there exists a unique $\mathbb{V}$-valued progressively measurable process $X^\e\in L^\infty([0,T],\mathbb{W})\cap D([0,T],\mathbb{V})$ such that for any $t>0$
\begin{eqnarray}\label{LDP origin eq}
X^\e(t)
&=&
X_0-\kappa\int_0^t\widehat{A}X^\e(s)ds-\int_0^t\widehat{B}(X^\e(s),X^\e(s))ds\nonumber\\
& &+\int_0^t\widehat{F}(X^\e(s),s)ds+\sqrt{\e}\int_0^t\widehat{G}(X^\e(s),s)dW(s)\\
& &+\e\int_0^t\int_{\mathbb{Z}}\widehat{\sigma}(s,X^\e(s-),z)\widetilde{N}^{\e^{-1}}(dzds)\ \text{ in }\mathbb{W}^\ast,  \ \ P-a.s.\nonumber
\end{eqnarray}
\end{thm}
\vskip0.2cm
\indent Theorem \ref{LDP Thm solution} shows that the above equation admits a strong solution in the probabilistic sense. In particular, for every $\e>0$, there exists a measurable map $\mathcal{G}^\e:C([0,T],\mathbb{R})\times M_{FC}([0,T]\times\mathbb{Z})\rightarrow D([0,T],\mathbb{V})$ such that, for any Poisson random measures $n^{\e^{-1}}$ on $[0,T]\times\mathbb{Z}$ with intensity measure $\e^{-1}\lambda_T\otimes\nu$, $\mathcal{G}^\e(\sqrt{\e}W,\e n^{\e^{-1}})$ is the unique solution of (\ref{LDP origin eq}) with $\widetilde{N}^{\e^{-1}}$ replaced by $\widetilde{n}^{\e^{-1}}$.
\begin{thm}\label{Theorem LDP}
Suppose that Condition \ref{condition 1} and \ref{condition 2} hold. Then the family $\{X^\e\}_{\e>0}$ satisfies a large deviation principle on $D([0,T],\mathbb{V})$ with a good rate function $I:D([0,T],\mathbb{V})\rightarrow[0,\infty]$, defined by
\begin{align*}
I(\xi):=\inf\left\{Q_1(f)+Q_2(g):\, \xi=\mathcal{G}^0\Big(\int_0^\cdot f(s)ds,\nu_T^g\Big),\,
f\in S_1^m, g\in S_2^m \text{ and }m\in\mathbb{N} \right\}.
\end{align*}
\end{thm}
\begin{proof}
Theorem \ref{LDP Thm solution} implies that for
each $\e>0$  there exists a mapping $\mathcal{G}^\e$ such that
\begin{align*}
\mathcal{G}^\e\big(\sqrt{\epsilon} W,\e N^{\e^{-1}}\big)
\stackrel{\mathcal{D}}{=} X^\e,
\end{align*}
where $X^\e$ is the solution of (\ref{1.a}) and $\stackrel{\mathcal{D}}{=}$ denotes equality in distribution.\\
\indent Define for each $m\in\mathbb{N}$ a space of stochastic processes on $\Omega$ by
\begin{align*}
\mathcal{S}_1^m:=\{\psi\colon [0,T]\times
\Omega\to H:\, {\mathbb{F}}\text{-predictable and} \, \psi(\cdot,\omega)\in S_1^m
\text{ for $P$-a.a. $\omega\in \Omega$}\}.
\end{align*}
Let $(K_n)_{n\in\mathbb{N}}$ be a sequence of compact sets $K_n\subseteq \mathbb{Z}$
with $ K_n \nearrow \mathbb{Z}$.  For each $n\in\mathbb{N}$, let
\begin{align*}
     \mathcal{A}_{b,n}
= \Big\{\psi\in \mathcal{A}:
\psi(t,z,\omega)\in \begin{cases}
 [\tfrac{1}{n},n], &\text{if }z\in K_n,\\
\{1\}, &\text{if }z\in K_n^c.
\end{cases}
\text{ for all }(t,\omega)\in [0,T]\times \Omega
\Big\},
\end{align*}
and let $\mathcal{A}_{b}=\bigcup _{n=1}^\infty \mathcal{A}_{b,n}$. Define
for each $m\in\mathbb{N}$ a space of stochastic process on $\Omega$ by
\begin{align*}
\mathcal{S}_2^m:=\{\varphi\in  \mathcal{A}_{b}:\, \varphi(\cdot,\cdot,\omega)\in S_2^m
\text{ for $P$-a.a. $\omega\in \Omega$}\}.
\end{align*}
According to Theorem 2.4 in \cite{BCD}, our claim is established once we have proved:
\begin{enumerate}
\item[(C1)] if $(f_n)_{n\in\mathbb{N}}\subseteq S_1^m$ converges to $f\in S_1^m$
and $(g_n)_{n\in\mathbb{N}}\subseteq S_2^m$ converges to $g\in S_2^m$ for some $m\in\mathbb{N}$,
then
      $$
         \mathcal{G}^0\Big(\int_{0}^{\cdot}f_{n}(s)\,ds,\, \nu_T^{g_n}\Big)\rightarrow \mathcal{G}^0\Big(\int_{0}^{\cdot}f(s)\, ds,
         \nu_T^g\Big)\quad\text{in }C([0,T],\mathbb{V}).
      $$
\item[(C2)] if $(\psi_\e)_{\e>0}\subseteq \mathcal{S}_1^m$ converges in distribution to $\psi\in \mathcal{S}_1^m$
and $(\varphi_\e)_{\e> 0}\subseteq \mathcal{S}_2^m$ converges in distribution to $\varphi\in \mathcal{S}_2^m$,
then
    $$
         \mathcal{G}^\e\Big(\sqrt{\epsilon} {W}(\cdot)+\int_{0}^{\cdot}\psi_{\e}(s)\, ds,\, \epsilon
          N^{\e^{-1}\varphi_\e}\Big) \text{ converges in distribution to }
         \mathcal{G}^0\Big(\int_{0}^{\cdot}\psi(s)\,ds,\,\nu_T^\varphi\Big)\text{ in } D([0,T],\mathbb{V}).
      $$
\end{enumerate}

\indent We give the details of the proof in the next section. (C1) will be given by Proposition \ref{LDP 1}. (C2) will be established by Proposition \ref{LDP 2}.

\end{proof}

\subsection{The proofs}


\begin{prp}\label{LDP 1}
Fix $m\in\mathbb{N}$, and let $q_n=(f_n,g_n),q=(f,g)\in S^m$ be such that $q_n\rightarrow q$ as $n\rightarrow\infty$. Then
\begin{eqnarray*}
\mathcal{G}^0(\int_0^\cdot f_n(s)ds,\nu_T^{g_n})\rightarrow\mathcal{G}^0(\int_0^\cdot f(s)ds,\nu_T^{g})\ \ \text{in} \ \ C([0,T],\mathbb{V}).
\end{eqnarray*}
\end{prp}
\begin{proof}
Recall $\mathcal{G}^0(\int_0^\cdot f_n(s)ds,\nu_T^{g_n})=\widetilde{X}^{q_n}$. For simplicity we denote $X_n=\widetilde{X}^{q_n}$.
Using similar arguments as (\ref{Estimation 2W}), (\ref{Estimation 2}), we can prove that there exists $C_{m}$ and $C_{\alpha,m}$ such that
\begin{eqnarray}\label{LDP 1 eatimation 1}
\sup_{n}\sup_{s\in[0,T]}\|X_n(s)\|_\mathbb{W}^{2}\leq C_{m},
\end{eqnarray}
and for $\alpha\in(0,\frac{1}{2})$
\begin{eqnarray*}
\sup_{n}\|X_n\|_{\mathbb{W}^{\alpha,2}([0,T],\mathbb{W}^*)}\leq C_{\alpha,m}.
\end{eqnarray*}
Hence, we can assert the existence of an element $X\in L^2([0,T],\mathbb{W})\cap L^\infty([0,T],\mathbb{V})$ and a sub-sequence $X_{m'}$ such that, as $m'\rightarrow\infty$\\
\indent (a) $\sup_{s\in[0,T]}\|X(s)\|_\mathbb{W}^{2}\leq C_{m}$,\\
\indent (b) $X_{m'}\rightarrow X$ in $L^2([0,T],\mathbb{W})$ weakly,\\
\indent (c) $X_{m'}\rightarrow X$ in $L^\infty([0,T],\mathbb{V})$ weak-star.\\
Combining Lemma \ref{H-lemma 4}, we have\\
\indent (d) $X_{m'}\rightarrow X$ in $L^2([0,T],\mathbb{V})$ strongly.
\vskip 0.2cm
We will prove $X=\widetilde{X}^q$.\\
\indent Let $\psi$ be a differentiable function on $[0,T]$ with $\psi(T)=0$. We multiply $X_{m'}(t)$ by $\psi(t)e_j$, and then integrate by parts to obtain
\begin{eqnarray}\label{Condition A 1}
& &-\int_0^T\big(X_{m'}(t),\psi'(t)e_j\big)_\mathbb{V}dt+\kappa\int_0^T\big(\big(X_{m'}(t),\psi(t)e_j\big)\big)dt\nonumber\\
&=&
\big(X_0,\psi(0)e_j\big)_\mathbb{V}-\int_0^T\big\langle\widehat{B}(X_{m'}(t),X_{m'}(t)),\psi(t)e_j\big\rangle_{\mathbb{W}^\ast,\mathbb{W}}dt\nonumber\\
& &+\int_0^T\big(\widehat{F}(X_{m'}(t),t),\psi(t)e_j\big)_\mathbb{V}dt+\int_0^T\big(\widehat{G}(X_{m'}(t),t)f_{m'}(t),\psi(t)e_j\big)_\mathbb{V}dt\nonumber\\
& &+\int_0^T\big(\int_{\mathbb{Z}}\widehat{\sigma}(t,X_{m'}(t),z)(g_{m'}(t,z)-1)\nu(dz),\psi(t)e_j\big)_\mathbb{V}dt.
\end{eqnarray}
Set
\begin{eqnarray*}
J^1_{m'}(T)=\int_0^T\big(\widehat{G}(X_{m'}(t),t)f_{m'}(t),\psi(t)e_j\big)_\mathbb{V}dt,
\end{eqnarray*}
\begin{eqnarray*}
J^2_{m'}(T)=\int_0^T\big(\widehat{G}(X(t),t)f_{m'}(t),\psi(t)e_j\big)_\mathbb{V}dt,
\end{eqnarray*}
\begin{eqnarray*}
J(T)=\int_0^T\big(\widehat{G}(X(t),t)f(t),\psi(t)e_j\big)_\mathbb{V}dt.
\end{eqnarray*}
Using $X_{m'}\rightarrow X$ in $L^2([0,T],\mathbb{V})$ strongly, we can easily get
\begin{eqnarray}\label{condition 1 limit 1}
\lim_{{m'}\rightarrow\infty}\big(J^1_{m'}(T)-J^2_{m'}(T)\big)=0.
\end{eqnarray}
Since $f_n\rightarrow f$ in $S_1^m$ and the linear mapping $: h\mapsto\int_0^T\big(\widehat{G}(X(t),t)h(t),\psi(t)e_j\big)_\mathbb{V}dt$ is strong continuous, we have
\begin{eqnarray}\label{condition 1 limit 2}
\lim_{{m'}\rightarrow\infty}\big(J^2_{m'}(T)-J(T)\big)=0.
\end{eqnarray}
Combining (\ref{condition 1 limit 1}), (\ref{condition 1 limit 2}), we get
\begin{eqnarray}\label{condition 1 limit 3}
\lim_{{m'}\rightarrow\infty}\big(J^1_{m'}(T)-J(T)\big)=0.
\end{eqnarray}

Set
\begin{eqnarray*}
I^1_{m'}(T)=\int_0^T\big(\int_{\mathbb{Z}}\widehat{\sigma}(t,X_{m'}(t),z)(g_{m'}(t,z)-1)\nu(dz),\psi(t)e_j\big)_\mathbb{V}dt,
\end{eqnarray*}
\begin{eqnarray*}
I^2_{m'}(T)=\int_0^T\big(\int_{\mathbb{Z}}\widehat{\sigma}(t,X(t),z)(g_{m'}(t,z)-1)\nu(dz),\psi(t)e_j\big)_\mathbb{V}dt,
\end{eqnarray*}
\begin{eqnarray*}
I(T)=\int_0^T\big(\int_{\mathbb{Z}}\widehat{\sigma}(t,X(t),z)(g(t,z)-1)\nu(dz),\psi(t)e_j\big)_\mathbb{V}dt.
\end{eqnarray*}
Similarly as the proof of (4.25)( see (4.26) and (4.29)) in Zhai and Zhang \cite{ZZ1}, we can get
\begin{eqnarray*}
\lim_{{m'}\rightarrow\infty}\sup_{h\in S_2^m}\int_0^T\int_{\mathbb{Z}}\|\widehat{\sigma}(t,X_{m'}(t),z)(h(t,z)-1)-\widehat{\sigma}(t,X(t),z)(h(t,z)-1)\|_{\mathbb{V}}\nu(dz)dt=0.
\end{eqnarray*}
Then, we can easily get
\begin{eqnarray}\label{condition 1 limit 4}
\lim_{{m'}\rightarrow\infty}\big(I^1_{m'}(T)-I^2_{m'}(T)\big)=0.
\end{eqnarray}
We know that
\begin{eqnarray*}
\|\widehat{\sigma}(t,X(t),z)\|_\mathbb{V}
\leq
\big(1+\sup_{s\in[0,T]}\|X(s)\|_\mathbb{V}\big)\|\sigma(t,z)\|_{0,\mathbb{V}}
\leq
C\|\sigma(t,z)\|_{0,\mathbb{V}},
\end{eqnarray*}
combining this with Remark \ref{remark}, we now get from Lemma \ref{H-lemma 3} that
\begin{eqnarray}\label{condition 1 limit 5}
\lim_{{m'}\rightarrow\infty}\big(I^2_{m'}(T)-I(T)\big)=0.
\end{eqnarray}
Combining (\ref{condition 1 limit 4}),(\ref{condition 1 limit 5}), we obtain
\begin{eqnarray}\label{condition 1 limit 6}
\lim_{{m'}\rightarrow\infty}\big(I^1_{m'}(T)-I(T)\big)=0.
\end{eqnarray}
By (\ref{Condition A 1}), (\ref{condition 1 limit 3}), (\ref{condition 1 limit 6}) and claim (a),(b),(c),(d), we can prove that $X$ satisfies
\begin{eqnarray}\label{Condition A 2}
& &-\int_0^T\big(X(t),\psi'(t)e_j\big)_\mathbb{V}dt+\kappa\int_0^T\big(\big(X(t),\psi(t)e_j\big)\big)dt\nonumber\\
&=&
\big(X_0,\psi(0)e_j\big)_\mathbb{V}-\int_0^T\big\langle\widehat{B}(X(t),X(t)),\psi(t)e_j\big\rangle_{\mathbb{W}^\ast,\mathbb{W}}dt\nonumber\\
& &+\int_0^T\big(\widehat{F}(X(t),t),\psi(t)e_j\big)_\mathbb{V}dt+\int_0^T\big(\widehat{G}(X(t),t)f(t),\psi(t)e_j\big)_\mathbb{V}dt\nonumber\\
& &+\int_0^T\big(\int_{\mathbb{Z}}\widehat{\sigma}(t,X(t),z)(g(t,z)-1)\nu(dz),\psi(t)e_j\big)_\mathbb{V}dt,
\end{eqnarray}
and using the same argument as in the proof of Theorem 3.1 in Temam \cite{Temam 1979}, Section 3, Chapter 3, we can conclude $X=\widetilde{X}^q$.

\vskip0.2cm

\indent Next, we prove $X_{m'}\rightarrow X$ in $C([0,T],\mathbb{V})$. Let $Z_{m'}=X_{m'}-X$. Then
\begin{eqnarray}\label{Condition A limit 0}
& &d\big\|Z_{m'}(s)\big\|_\mathbb{V}^2/ds+2\kappa\big\|Z_{m'}(s)\big\|^2\nonumber\\
&=&
-2\big\langle\widehat{B}(X_{m'}(s),X_{m'}(s))-\widehat{B}(X(s),X(s)),X_{m'}(s)-X(s)\big\rangle_{\mathbb{W}\ast,\mathbb{W}}\nonumber\\
& &+2\big(\widehat{F}(X_{m'}(s),s)-\widehat{F}(X(s),s),X_{m'}(s)-X(s)\big)_\mathbb{V}\nonumber\\
& &+2\big(\widehat{G}(X_{m'}(s),s)f_{m'}(s)-\widehat{G}(X(s),s)f(s),X_{m'}(s)-X(s)\big)_\mathbb{V}\nonumber\\
& &+2\int_\mathbb{Z}\big(\widehat{\sigma}(s,X_{m'}(s),z)(g_{m'}(s,z)-1)-\widehat{\sigma}(s,X(s),z)(g(s,z)-1),X_{m'}(s)-X(s)\big)_\mathbb{V}\nu(dz)\nonumber\\
&:=&I^1_{m'}(s)+I^2_{m'}(s)+I^3_{m'}(s)+I^4_{m'}(s).
\end{eqnarray}
By Lemma \ref{Lem-B-01} and claim (a), we have
\begin{eqnarray}\label{Condition A limit 1}
|I^1_{m'}(s)|
&\leq&
2\big|\big\langle\widehat{B}(X_{m'}(s)-X(s),X_{m'}(s)-X(s)),X(s)\big\rangle\big|\nonumber\\
&\leq&
C_B\|X_{m'}(s)-X(s)\|_\mathbb{V}^2\|X(s)\|_\mathbb{W}\nonumber\\
&\leq&
C\|X_{m'}(s)-X(s)\|_\mathbb{V}^2.
\end{eqnarray}
By Condition \ref{condition 1}, we have
\begin{eqnarray}\label{Condition A limit 2}
|I^2_{m'}(s)|
&\leq&
2\|\widehat{F}(X_{m'}(s),s)-\widehat{F}(X(s),s)\|_\mathbb{V}\|X_{m'}(s)-X(s)\|_\mathbb{V}\nonumber\\
&\leq&
C_F\|X_{m'}(s)-X(s)\|_\mathbb{V}^2,
\end{eqnarray}
and
\begin{eqnarray}\label{Condition A limit 3}
\int_0^T|I^3_{m'}(s)|ds
&\leq&
2\int_0^T\big|\big(\widehat{G}(X_{m'}(s),s)f_{m'}(s)-\widehat{G}(X(s),s)f_{m'}(s),X_{m'}(s)-X(s)\big)_\mathbb{V}\big|ds\nonumber\\
& &+2\int_0^T\big|\big(\widehat{G}(X(s),s)f_{m'}(s)-\widehat{G}(X(s),s)f(s),X_{m'}(s)-X(s)\big)_\mathbb{V}\big|ds\nonumber\\
&\leq&
C_G\int_0^T|f_{m'}(s)|\|X_{m'}(s)-X(s)\|_\mathbb{V}^2ds\nonumber\\
& &+C_G\int_0^T|f_{m'}(s)-f(s)|\|X_{m'}(s)-X(s)\|_\mathbb{V}\|X(s)\|_\mathbb{V}ds\nonumber\\
&\leq&
C_G\int_0^T(1+|f_{m'}(s)|^2)|\|X_{m'}(s)-X(s)\|_\mathbb{V}^2ds+\Upsilon_{m'}^1(T),
\end{eqnarray}
where $\Upsilon_{m'}^1(T)=C_G\big(\sup_{s\in[0,T]}\|X(s)\|_\mathbb{V}\big)\big(\int_0^T|f_{m'}(s)-f(s)|^2ds\big)^{\frac{1}{2}}
\big(\int_0^T\|X_{m'}(s)-X(s)\|_\mathbb{V}^2ds\big)^{\frac{1}{2}}$, claim (d) and $f_{m'},f\in S_1^m$ imply that
\begin{eqnarray}\label{Condition A limit g1}
\lim_{m'\rightarrow\infty}\Upsilon_{m'}^1(T)=0.
\end{eqnarray}
For $I^4_{m'}(s)$, we have
\begin{eqnarray}\label{Condition A limit 4}
& &\int_0^TI^4_{m'}(s)ds\nonumber\\
&\leq&
2\int_0^T\int_{\mathbb{Z}}\|\sigma(s,X_{m'}(s),z)\|_\mathbb{V}|g_{m'}(s,z)-1|\|X_{m'}(s)-X(s)\|_\mathbb{V}\nu(dz)ds\nonumber\\
& &+2\int_0^T\int_{\mathbb{Z}}\|\sigma(s,X(s),z)\|_\mathbb{V}|g(s,z)-1|\|X_{m'}(s)-X(s)\|_\mathbb{V}\nu(dz)ds\nonumber\\
&\leq&
2\int_0^T\int_{\mathbb{Z}}\|\sigma(s,z)\|_{0,\mathbb{V}}(1+\|X_{m'}(s)\|_\mathbb{V})|g_{m'}(s,z)-1|\|X_{m'}(s)-X(s)\|_\mathbb{V}\nu(dz)ds\nonumber\\
& &+2\int_0^T\int_{\mathbb{Z}}\|\sigma(s,z)\|_{0,\mathbb{V}}(1+\|X(s)\|_\mathbb{V})|g(s,z)-1|\|X_{m'}(s)-X(s)\|_\mathbb{V}\nu(dz)ds\nonumber\\
&\leq&
2(1+C_m)\int_0^T\int_{\mathbb{Z}}\|\sigma(t,z)\|_{0,\mathbb{V}}|g_{m'}(s,z)-1|\|X_{m'}(s)-X(s)\|_\mathbb{V}\nu(dz)ds\nonumber\\
& &+2(1+C_m)\int_0^T\int_{\mathbb{Z}}\|\sigma(s,z)\|_{0,\mathbb{V}}|g(s,z)-1|\|X_{m'}(s)-X(s)\|_\mathbb{V}\nu(dz)ds\nonumber\\
&:=&
\Upsilon_{m'}^2(T).
\end{eqnarray}
Together with (\ref{Condition A limit 0}), (\ref{Condition A limit 1}), (\ref{Condition A limit 2}), (\ref{Condition A limit 3}), (\ref{Condition A limit 4}), we obtain
\begin{eqnarray}\label{Condition A limit}
\|Z_{m'}(t)\|_\mathbb{V}^2+\kappa\int_0^t\|Z_{m'}(s)\|_\mathbb{V}^2ds
&\leq&
\int_0^t\Psi_{m'}(s)\|Z_{m'}(s)\|_\mathbb{V}^2ds+\Upsilon_{m'}^1(T)+\Upsilon_{m'}^2(T),
\end{eqnarray}
where $\Psi_{m'}(s)=C+C(1+f_{m'}^2(s))$, satisfying $\sup_{m'}\int_0^T\Psi_{m'}(s)ds<\infty.$ \\
Then by Gronwall's inequality, we get
\begin{eqnarray*}
\sup_{t\in[0,T]}\|Z_{m'}(t)\|_\mathbb{V}^2
\leq
(\Upsilon_{m'}^1(T)+\Upsilon_{m'}^2(T))\text{exp}\{\sup_{m'}\int_0^T\Psi_{m'}(s)ds\}.
\end{eqnarray*}
By (4.29) in Zhai and Zhang \cite{ZZ1}, we can easily get $\Upsilon_{m'}^2(T)\rightarrow0$ as $m'\rightarrow\infty$, combining (\ref{Condition A limit g1}), we get
\begin{eqnarray*}
\lim_{m'\rightarrow\infty}\sup_{t\in[0,T]}\|Z_{m'}(t)\|_\mathbb{V}^2=0.
\end{eqnarray*}
The proof is completed.
\end{proof}

\vskip 0.2cm

\indent Let $\phi_\e=(\psi_\e,\varphi_\e)\in\mathcal{S}_1^m\times\mathcal{S}_2^m$ and $\vartheta_\e=\frac{1}{\varphi_\e}$. The following lemma was stated in Budhiraja, Dupuis and Maroulas \cite{BDM}( see Lemma 2.3 there).
\begin{lem}\label{LDP Girsanov}
\begin{eqnarray*}
\mathcal{E}_t^\e(\vartheta_\e)
&:=&
\exp\Big\{\int_{[0,t]\times\mathbb{Z}\times[0,\e^{-1}\varphi_\e]}\log\big(\vartheta_\e(s,z)\big)N(ds,dz,dr)\nonumber\\
& &\qquad+\int_{[0,t]\times\mathbb{Z}\times[0,\e^{-1}\varphi_\e]}\big(-\vartheta_\e(s,z)+1\big)ds\nu(dz)dr\Big\}
\end{eqnarray*}
and
\begin{eqnarray*}
\bar{\mathcal{E}}_t^\e(\psi_\e):=\exp\Big\{\frac{1}{\sqrt{\e}}\int_0^t\psi_\e(s)dW(s)-\frac{1}{2\e}\int_0^t\|\psi_\e(s)\|^2ds\Big\}
\end{eqnarray*}
are $\mathbb{F}$-martingales. Set
\begin{eqnarray*}
\tilde{\mathcal{E}}_t^\e(\psi_\e,\vartheta_\e):=\bar{\mathcal{E}}_t^\e(\psi_\e)\mathcal{E}_t^\e(\vartheta_\e).
\end{eqnarray*}
Then
\begin{eqnarray*}
Q_t^\e(G)=\int_{G}{\tilde{\mathcal{E}}}_t^\e(\psi_\e,\vartheta_\e)d{P} \ \ \text{for  }G\in\mathcal{B}(C([0,T],\mathbb{R})\otimes \mathcal{T}\big(M_{FC}([0,T]\times\mathbb{Z})\big)
\end{eqnarray*}
defines a probability measure on $C([0,T],\mathbb{R})\times M_{FC}([0,T]\times\mathbb{Z})$.
\end{lem}

\indent Since $(\sqrt{\e}W(\cdot)+\int_0^{\cdot}\psi_\e(s)ds,\e N^{{\e}^{-1}\varphi_\e})$ under $Q_T^\e$ has the same law as that of $(\sqrt{\e}W,\e N^{\e^{-1}})$ under $P$, by Theorem \ref{LDP Thm solution} it follows that there exists a unique solution to the following controlled stochastic evolution equation, denoted by $\widetilde{X}^\e$:
\begin{eqnarray}\label{LDP Girsanov eq}
\widetilde{X}^\e(t)
&=&
X_0-\kappa\int_0^t\widehat{A}\widetilde{X}^\e(s)ds-\int_0^t\widehat{B}(\widetilde{X}^\e(s),\widetilde{X}^\e(s))ds
+\int_0^t\widehat{F}(\widetilde{X}^\e(s),s)ds\nonumber\\
& &+\int_0^t\widehat{G}(\widetilde{X}^\e(s),s)\psi_\e(s)ds+\sqrt{\e}\int_0^t\widehat{G}(\widetilde{X}^\e(s),s)dW(s)\nonumber\\
& &+\int_0^t\int_{\mathbb{Z}}\widehat{\sigma}(s,\widetilde{X}^\e(s-),z)\big(\e N^{\e^{-1}\varphi_\e}(dzds)-\nu(dz)ds\big)\nonumber\\
&=&
X_0-\kappa\int_0^t\widehat{A}\widetilde{X}^\e(s)ds-\int_0^t\widehat{B}(\widetilde{X}^\e(s),\widetilde{X}^\e(s))ds
+\int_0^t\widehat{F}(\widetilde{X}^\e(s),s)ds\nonumber\\
& &+\int_0^t\widehat{G}(\widetilde{X}^\e(s),s)\psi_\e(s)ds+\sqrt{\e}\int_0^t\widehat{G}(\widetilde{X}^\e(s),s)dW(s)\nonumber\\
& &+\int_0^t\int_{\mathbb{Z}}\widehat{\sigma}(s,\widetilde{X}^\e(s),z)(\varphi_\e(s,z)-1)\nu(dz)ds\nonumber\\
& &+\e\int_0^t\int_{\mathbb{Z}}\widehat{\sigma}(s,\widetilde{X}^\e(s-),z)\widetilde{N}^{\e^{-1}\varphi_\e}(dzds),
\end{eqnarray}
and we have
\begin{eqnarray}\label{LDP Girsanov rep}
\mathcal{G}^\e\big(\sqrt{\e}W(\cdot)+\int_0^{\cdot}\psi_\e(s)ds,\e N^{{\e}^{-1}\varphi_\e}\big)\stackrel{\mathcal{D}}{=}\widetilde{X}^\e.
\end{eqnarray}
The following estimates will be used later.

\begin{lem}\label{LDP Lemma Estimates 1}
There exists $\e_0>0$ such that

\begin{eqnarray}\label{LDP Estimate 1W}
\sup_{0<\e<\e_0}E\Big[\sup_{t\in[0,T]}\|\widetilde{X}^\e(t)\|_\mathbb{W}^{2}\Big]\leq C<\infty.
\end{eqnarray}

\end{lem}

\begin{proof}
Define
$$\tau_M=\inf\{t\geq0:\|\widetilde{X}^{\e}(t)\|_\mathbb{W}\geq M\}\wedge T.$$
\indent Multiplying $\lambda_i$ at both sides of the equation (\ref{LDP Girsanov eq}), we can use (\ref{Basis}) to obtain
\begin{eqnarray}\label{LDP Estimate eq 1}
& &d\big(\widetilde{X}^{\e}(s),e_i\big)_\mathbb{W}
+\kappa\big(\widehat{A}\widetilde{X}^{\e}(s),e_i\big)_{\mathbb{W}}ds\nonumber\\
&=&
-\big(\widehat{B}(\widetilde{X}^{\e}(s),\widetilde{X}^{\e}(s)),e_i\big)_{\mathbb{W}}ds
+\big(\widehat{F}(\widetilde{X}^{\e}(s),s),e_i\big)_{\mathbb{W}}ds\nonumber\\
& &
+\sqrt{\e}\big(\widehat{G}(\widetilde{X}^{\e}(s),s),e_i\big)_{\mathbb{W}}dW(s)
+\big(\widehat{G}(\widetilde{X}^{\e}(s),s)\psi_\e(s),e_i\big)_{\mathbb{W}}ds\nonumber\\
& &
+\big(\int_{\mathbb{Z}}\widehat{\sigma}(s,\widetilde{X}^{\e}(s),z)(\varphi_\e(s,z)-1)\nu(dz),e_i\big)_{\mathbb{W}}ds\nonumber\\
& &
+\e\int_{\mathbb{Z}}\big(\widehat{\sigma}(s,\widetilde{X}^{\e}(s-),z),e_i\big)_{\mathbb{W}}\widetilde{N}^{\e^{-1}\varphi_\e}(dzds),
\end{eqnarray}
for any $i\in\mathbb{N}$.\\
\indent Applying $It\hat{o}$ formula to $\big(\widetilde{X}^{\e}(s),e_i\big)_\mathbb{W}^2$ and then summing over $i$ from $1$ to $\infty$ yields
\begin{eqnarray}\label{LDP Estimate eq 2}
\|\widetilde{X}^{\e}(t)\|_\mathbb{W}^2
&=&
\|X_0\|_\mathbb{W}^2-2\kappa\int_0^t\big(\widehat{A}\widetilde{X}^{\e}(s),\widetilde{X}^{\e}(s)\big)_{\mathbb{W}}ds\nonumber\\
& &
-2\int_0^t\big(\widehat{B}(\widetilde{X}^{\e}(s),\widetilde{X}^{\e}(s)),\widetilde{X}^{\e}(s)\big)_{\mathbb{W}}ds
+2\int_0^t\big(\widehat{F}(\widetilde{X}^{\e}(s),s),\widetilde{X}^{\e}(s)\big)_{\mathbb{W}}ds\nonumber\\
& &
+2\sqrt{\e}\int_0^t\big(\widehat{G}(\widetilde{X}^{\e}(s),s),\widetilde{X}^{\e}(s)\big)_{\mathbb{W}}dW(s)
+\e\int_0^t\|\widehat{G}(\widetilde{X}^{\e}(s),s)\|_{\mathbb{W}}^2ds\nonumber\\
& &
+2\int_0^t\big(\widehat{G}(\widetilde{X}^{\e}(s),s)\psi_\e(s),\widetilde{X}^{\e}(s)\big)_{\mathbb{W}}ds\nonumber\\
& &
+2\int_0^t\big(\int_{\mathbb{Z}}
\widehat{\sigma}(s,\widetilde{X}^{\e}(s),z)(\varphi_\e(s,z)-1)\nu(dz),\widetilde{X}^{\e}(s)\big)_{\mathbb{W}}ds\nonumber\\
& &
+2\e\int_0^t\int_{\mathbb{Z}}\big(
\widehat{\sigma}(s,\widetilde{X}^{\e}(s-),z),\widetilde{X}^{\e}(s-)\big)_{\mathbb{W}}\widetilde{N}^{\e^{-1}\varphi_\e}(dzds)\nonumber\\
& &
+\int_0^t\int_{\mathbb{Z}}\|\e\widehat{\sigma}(s,\widetilde{X}^{\e}(s-),z)\|_{\mathbb{W}}^2N^{\e^{-1}\varphi_\e}(dzds).
\end{eqnarray}
By a simple calculus, refer to (\ref{Estimation 2W eq 4}), we get
\begin{eqnarray}\label{LDP Estimate eq 3}
& &\|\widetilde{X}^{\e}(t)\|_\mathbb{W}^2
+\frac{2\kappa}{\alpha}\int_0^t\|\widetilde{X}^{\e}(s)\|_\mathbb{W}^2ds\nonumber\\
&=&
\|X_0\|_\mathbb{W}^2+\frac{2\kappa}{\alpha}\int_0^t
\Big(curl\big(\widetilde{X}^{\e}(s)\big),curl\big(\widetilde{X}^{\e}(s)-\alpha\Delta\widetilde{X}^{\e}(s)\big)\Big)ds\nonumber\\
& &
+2\int_0^t\big(\widehat{F}(\widetilde{X}^{\e}(s),s),\widetilde{X}^{\e}(s)\big)_{\mathbb{W}}ds
+2\sqrt{\e}\int_0^t\big(\widehat{G}(\widetilde{X}^{\e}(s),s),\widetilde{X}^{\e}(s)\big)_{\mathbb{W}}dW(s)\nonumber\\
& &
+\e\int_0^t\|\widehat{G}(\widetilde{X}^{\e}(s),s)\|_{\mathbb{W}}^2ds
+2\int_0^t\big(\widehat{G}(\widetilde{X}^{\e}(s),s)\psi_\e(s),\widetilde{X}^{\e}(s)\big)_{\mathbb{W}}ds\nonumber\\
& &
+2\int_0^t\big(\int_{\mathbb{Z}}
\widehat{\sigma}(s,\widetilde{X}^{\e}(s),z)(\varphi_\e(s,z)-1)\nu(dz),\widetilde{X}^{\e}(s)\big)_{\mathbb{W}}ds\nonumber\\
& &
+2\e\int_0^t\int_{\mathbb{Z}}\big(
\widehat{\sigma}(s,\widetilde{X}^{\e}(s-),z),\widetilde{X}^{\e}(s-)\big)_{\mathbb{W}}\widetilde{N}^{\e^{-1}\varphi_\e}(dzds)\nonumber\\
& &
+\e^2\int_0^t\int_{\mathbb{Z}}\|\widehat{\sigma}(s,\widetilde{X}^{\e}(s-),z)\|_{\mathbb{W}}^2N^{\e^{-1}\varphi_\e}(dzds)\nonumber\\
&:=&
\|X_0\|_\mathbb{W}^2+I_1(t)+I_2(t)+I_3(t)+I_4(t)+I_5(t)+I_6(t),
\end{eqnarray}
where
\begin{eqnarray*}
I_1(t)
&:=&
\frac{2\kappa}{\alpha}\int_0^t
\Big(curl\big(\widetilde{X}^{\e}(s)\big),curl\big(\widetilde{X}^{\e}(s)-\alpha\Delta\widetilde{X}^{\e}(s)\big)\Big)ds\nonumber\\
& &
+2\int_0^{t}\big(\widehat{F}(\widetilde{X}^{\e}(s),s),\widetilde{X}^{\e}(s)\big)_{\mathbb{W}}ds
    +\e\int_0^{t}\|\widehat{G}(\widetilde{X}^{\e}(s),s)\|_{\mathbb{W}}^2ds,\nonumber\\
I_2(t)
&:=&
2\sqrt{\e}\int_0^t\big(\widehat{G}(\widetilde{X}^{\e}(s),s),\widetilde{X}^{\e}(s)\big)_{\mathbb{W}}dW(s),\nonumber\\
I_3(t)
&:=&
2\int_0^{t}\big(\widehat{G}(\widetilde{X}^{\e}(s),s)\psi_\e(s),\widetilde{X}^{\e}(s)\big)_{\mathbb{W}}ds,\nonumber\\
I_4(t)
&:=&
2\int_0^{t}\big(\int_{\mathbb{Z}}
\widehat{\sigma}(s,\widetilde{X}^{\e}(s),z)(\varphi_\e(s,z)-1)\nu(dz),\widetilde{X}^{\e}(s)\big)_{\mathbb{W}}ds,\nonumber\\
I_5(t)
&:=&
2\e\int_0^t\int_{\mathbb{Z}}\big(
    \widehat{\sigma}(s,\widetilde{X}^{\e}(s-),z),\widetilde{X}^{\e}(s-)\big)_{\mathbb{W}}\widetilde{N}^{\e^{-1}\varphi_\e}(dzds),\nonumber\\
I_6(t)
&:=&
{\e}^2\int_0^{t}\int_{\mathbb{Z}}
    \|\widehat{\sigma}(s,\widetilde{X}^{\e}(s-),z)\|_{\mathbb{W}}^2N^{\e^{-1}\varphi_\e}(dzds).
\end{eqnarray*}

\noindent By Condition \ref{condition 1} and (\ref{inverse operator transform}), (\ref{curl}), we have

\begin{eqnarray}\label{LDP Estimate eq 4}
I_1(t)
&\leq&
C \int_0^{t}\|\widetilde{X}^{\e}(s)\|_\mathbb{W}\|\widetilde{X}^{\e}(s)\|_\mathbb{V}ds
+C \int_0^{t}\|\widetilde{X}^{\e}(s)\|_\mathbb{W}\|F(\widetilde{X}^{\e}(s),s)\|_\mathbb{V}ds
+C \int_0^{t}\|G(\widetilde{X}^{\e}(s),s)\|_\mathbb{V}^2ds\nonumber\\
&\leq&
C\int_0^t \|\widetilde{X}^{\e}(s)\|_\mathbb{W}^2ds.
\end{eqnarray}
\noindent By B-D-G and Young's inequalities and Lemma \ref{H-lemma 1}, we get
\begin{eqnarray}\label{LDP Estimate eq 5}
E\Big[\sup_{t\in[0,T\wedge\tau_M]}I_2(t)\Big]
&\leq&
C\sqrt{\e}E\Big[\int_0^{T\wedge\tau_M}
        \big(\widehat{G}(\widetilde{X}^{\e}(s),s),\widetilde{X}^{\e}(s)\big)_{\mathbb{W}}^2ds\Big]^{\frac{1}{2}}\nonumber\\
&\leq&
C\sqrt{\e}E\Big[\Big(\sup_{t\in[0,T\wedge\tau_M]}\|\widetilde{X}^{\e}(t)\|_\mathbb{W}^2\Big)^{\frac{1}{2}}\Big(\int_0^{T\wedge\tau_M}
        \|\widetilde{X}^{\e}(s)\|_\mathbb{W}^2ds\Big)^{\frac{1}{2}}\Big]\nonumber\\
&\leq&
C\sqrt{\e}E\Big[\sup_{t\in[0,T\wedge\tau_M]}\|\widetilde{X}^{\e}(t)\|_\mathbb{W}^2\Big]
+C\sqrt{\e}\int_0^T E\Big[\sup_{s\in[0,t\wedge\tau_M]}\|\widetilde{X}^{\e}(s)\|_\mathbb{W}^2\Big]dt\nonumber\\
&\leq&
C\sqrt{\e}E\Big[\sup_{t\in[0,T\wedge\tau_M]}\|\widetilde{X}^{\e}(t)\|_\mathbb{W}^2\Big],
\end{eqnarray}
and
\begin{eqnarray}\label{LDP Estimate eq 6}
E\Big[\sup_{t\in[0,T\wedge\tau_M]}I_5(t)\Big]
&\leq&
2\e E\Big[\Big(\int_0^{T\wedge\tau_M}
\int_{\mathbb{Z}}\big(\widehat{\sigma}(s,\widetilde{X}^{\e}(s-),z),
\widetilde{X}^{\e}(s-)\big)_{\mathbb{W}}^2N^{\e^{-1}\varphi_\e}(dzds)\Big)^\frac{1}{2}\Big]\nonumber\\
&\leq&
2\e E\Big[\Big(\sup_{s\in[0,T\wedge\tau_M]}\|\widetilde{X}^{\e}(s)\|_\mathbb{W}^2\Big)^\frac{1}{2}
\Big(\int_0^T\int_{\mathbb{Z}}\|\widehat{\sigma}(s,\widetilde{X}^{\e}(s-),z)\|_{\mathbb{W}}^2N^{\e^{-1}\varphi_\e}(dzds)\Big)^\frac{1}{2}\Big]\nonumber\\
&\leq&
{\e}^{\frac{2}{3}}E\Big[\sup_{t\in[0,T\wedge\tau_M]}\|\widetilde{X}^{\e}(t)\|_\mathbb{W}^2\Big]\nonumber\\
&  &
\qquad+{\e}^{\frac{1}{3}}CE\Big[\Big(\int_0^{T\wedge\tau_M}\int_{\mathbb{Z}}
    \|{\sigma}(s,\widetilde{X}^{\e}(s),z)\|_{\mathbb{V}}^2\varphi_\e(s,z)\nu(dz)ds\Big)\Big]\nonumber\\
&\leq&
{\e}^{\frac{2}{3}}E\Big[\sup_{t\in[0,T\wedge\tau_M]}\|\widetilde{X}^{\e}(t)\|_\mathbb{W}^2\Big]+{\e}^{\frac{1}{3}}C\cdot C_{0,2}^{m}E\Big[1+\sup_{t\in[0,t\wedge\tau_M]}\|\widetilde{X}^{\e}(t)\|_\mathbb{W}^2\Big].
\end{eqnarray}
By Condition \ref{condition 1} and Lemma \ref{H-lemma 1}, we get
\begin{eqnarray}\label{LDP Estimate eq 7}
I_3(t)
&\leq&
C \int_0^{t}\|\widetilde{X}^{\e}(s)\|_\mathbb{W}^2|\psi_\e(s)|ds
\leq
C\int_0^t \|\widetilde{X}^{\e}(s)\|_\mathbb{W}^2\big(1+|\psi_\e(s)|^2\big)ds,
\end{eqnarray}
\begin{eqnarray}\label{LDP Estimate eq 8}
I_4(t)
&\leq&
C \int_0^{t}\int_{\mathbb{Z}}\|\widetilde{X}^{\e}(s)\|_\mathbb{W}
\|\sigma(s,\widetilde{X}^{\e}(s),z)\|_{\mathbb{V}}|\varphi_\e(s,z)-1|\nu(dz)ds\nonumber\\
&\leq&
C \int_0^{t}\int_{\mathbb{Z}}\big(1+\|\widetilde{X}^{\e}(s)\|_\mathbb{W}^2\big)
\|\sigma(s,z)\|_{0,\mathbb{V}}|\varphi_\e(s,z)-1|\nu(dz)ds\nonumber\\
&\leq&
C\cdot C_{0,1}^m+C\int_0^t \|\widetilde{X}^{\e}(s)\|_\mathbb{W}^2
\Big(\int_{\mathbb{Z}}\|\sigma(s,z)\|_{0,\mathbb{V}}|\varphi_\e(s,z)-1|\nu(dz)\Big)ds,
\end{eqnarray}
and
\begin{eqnarray}\label{LDP Estimate eq 9}
E\Big[\sup_{t\in[0,T\wedge\tau_M]}I_6(t)\Big]
&\leq&
{\e}^2E\Big[\int_0^{T\wedge\tau_M}\int_{\mathbb{Z}}
    \|\widehat{\sigma}(s,\widetilde{X}^{\e}(s-),z)\|_{\mathbb{W}}^2N^{\e^{-1}\varphi_\e}(dzds)\Big]\nonumber\\
&\leq&
\e C E\Big[\int_0^{T\wedge\tau_M}\int_{\mathbb{Z}}
\|\sigma(s,z)\|_{0,\mathbb{V}}^2\big(1+\|\widetilde{X}^{\e}(s)\|_\mathbb{V}^2\big)\varphi_\e(s,z)\nu(dz)ds\Big]\nonumber\\
&\leq&
\e C\cdot C_{0,2}^{m}\Big(1+E\Big[\sup_{t\in[0,T\wedge\tau_M]}\|\widetilde{X}^{\e}(t)\|_\mathbb{V}^2\Big]\Big).
\end{eqnarray}
Combining the estimates (\ref{LDP Estimate eq 4}), (\ref{LDP Estimate eq 7}) and (\ref{LDP Estimate eq 8}), we have
\begin{eqnarray}\label{LDP Estimate eq 10}
& &\|\widetilde{X}^{\e}(t)\|_\mathbb{W}^2
+\frac{2\kappa}{\alpha}\int_0^t\|\widetilde{X}^{\e}(s)\|_\mathbb{W}^2ds\nonumber\\
&\leq&
\Big(\|X_0\|_\mathbb{W}^2+C\cdot C_{0,1}^m+\sup_{s\in[0,t]}I_2(s)+\sup_{s\in[0,t]}I_5(s)+\sup_{s\in[0,t]}I_6(s)\Big)\nonumber\\
& &
+C\int_0^t\|\widetilde{X}^{\e}(s)\|_\mathbb{W}^2
\big(1+|\psi_\e(s)|^2+\int_{\mathbb{Z}}\|\sigma(s,z)\|_{0,\mathbb{V}}|\varphi_\e(s,z)-1|\nu(dz)\big)ds.
\end{eqnarray}
By Lemma \ref{H-lemma 1} and Gronwall's inequality, we get
\begin{eqnarray*}
\|\widetilde{X}^{\e}(t)\|_\mathbb{W}^2
&\leq&
\Big(\|X_0\|_\mathbb{W}^2+C\cdot C_{0,1}^m+\sup_{s\in[0,t]}I_2(s)+\sup_{s\in[0,t]}I_5(s)+\sup_{s\in[0,t]}I_6(s)\Big)
e^{C\big(1+m+C_{0,1}^m\big)}.
\end{eqnarray*}
Set $C_0=e^{C\big(1+m+C_{0,1}^m\big)}$. By (\ref{LDP Estimate eq 5}), (\ref{LDP Estimate eq 6}) and (\ref{LDP Estimate eq 9}), we have
\begin{eqnarray*}
\Big(1-C_0\big(\e C\cdot C_{0,2}^m+{\e}^{\frac{1}{3}}C\cdot C_{0,2}^m+C\sqrt{\e}+{\e}^{\frac{2}{3}}\big)\Big)
E\Big[\sup_{t\in[0,T\wedge\tau_M]}\|\widetilde{X}^{\e}(t)\|_\mathbb{W}^2\Big]
&\leq&
C_0\Big(\|X_0\|_\mathbb{W}^2+C\cdot C_{0,1}^m+\e C\cdot C_{0,2}^m\big).
\end{eqnarray*}
Since $C_0,C,C_{0,1}^m,C_{0,2}^m$ are constant independent of $\e$, we can select $\e_0$ small enough, such that
\begin{eqnarray*}
C_0\big(\e C\cdot C_{0,2}^m+{\e}^{\frac{1}{3}}C\cdot C_{0,2}^m+C\sqrt{\e}+{\e}^{\frac{2}{3}}\big)<\frac{1}{2},\ \ \forall\e\in(0,\e_0),
\end{eqnarray*}
then letting $M\rightarrow\infty$, we get (\ref{LDP Estimate 1W}).

\end{proof}

\indent By Proposition 3.1 in \cite{RZ}, there exists a unique solution $\widetilde{Y}^\e(t),t\geq0$ to the followig equation:
\begin{eqnarray}\label{LDP Y eq}
d\widetilde{Y}^\e(t)
&=&
-\kappa\widehat{A}\widetilde{Y}^\e(t)dt+\sqrt{\e}\widehat{G}(\widetilde{X}^\e(t),t)dW(t)\nonumber\\
& &
+\e\int_{\mathbb{Z}}\widehat{\sigma}(t,\widetilde{X}^\e(t),z){\widetilde{N}}^{\e^{-1}\varphi_\e}(dzdt)
\end{eqnarray}
with initial value $\widetilde{Y}^\e(0)=0$. Moreover, $\widetilde{Y}^\e\in D([0,T],\mathbb{V})\cap L^2([0,T],\mathbb{W})$, P-a.s. and we have the following estimates.
\begin{lem}\label{LDP Lemma Estimates 3}
There exists $C>0$ such that
\begin{eqnarray}\label{LDP Estimate 3W}
E\Big[\sup_{t\in[0,T]}\|\widetilde{Y}^{\e}(t)\|_\mathbb{W}^2\Big]\leq\e C.
\end{eqnarray}
\end{lem}




\begin{proof}
Define
$$\tau_M=\inf\{t\geq0,\|\widetilde{Y}^{\e}(t)\|_\mathbb{W}\geq M\}\wedge T.$$
\indent By (\ref{LDP Y eq}), we have
\begin{eqnarray}\label{LDP Estimate 4W eq 1}
d\big(\widetilde{Y}^{\e}(t),e_i\big)_{\mathbb{V}}
&=&
-\kappa\big(\widehat{A}\widetilde{Y}^{\e}(t),e_i\big)_{\mathbb{V}}dt
+\sqrt{\e}\big(\widehat{G}(\widetilde{X}^{\e}(t),t),e_i\big)_{\mathbb{V}}dW(t)\nonumber\\
& &
+\e\int_{\mathbb{Z}}\big(\widehat{\sigma}(t,\widetilde{X}^{\e}(t-),z),e_i\big)_{\mathbb{V}}{\widetilde{N}}^{\e^{-1}\varphi_\e}(dzdt),
\end{eqnarray}
for any $i\in\mathbb{N}$.\\
\indent Multiplying both sides of the equation (\ref{LDP Estimate 4W eq 2}) by $\lambda_i$, we can use (\ref{Basis}) to obtain
\begin{eqnarray}\label{LDP Estimate 4W eq 2}
d\big(\widetilde{Y}^{\e}(t),e_i\big)_{\mathbb{W}}
&=&
-\kappa\big(\widehat{A}\widetilde{Y}^{\e}(t),e_i\big)_{\mathbb{W}}dt\nonumber\\
& &
+\sqrt{\e}\big(\widehat{G}(\widetilde{X}^{\e}(t),t),e_i\big)_{\mathbb{W}}dW(t)\nonumber\\
& &
+\e\int_{\mathbb{Z}}\big(\widehat{\sigma}(t,\widetilde{X}^{\e}(t-),z),e_i\big)_{\mathbb{W}}{\widetilde{N}}^{\e^{-1}\varphi_\e}(dzdt),
\end{eqnarray}
for any $i\in\mathbb{N}$.\\
\indent Applying $It\hat{o}$ formula to $\big(\widetilde{Y}^{\e}(t),e_i\big)_{\mathbb{W}}^2$ and summing over $i$ from $1$ to $\infty$ yields
\begin{eqnarray}\label{LDP Estimate 4W eq 3}
\|\widetilde{Y}^{\e}(t)\|_{\mathbb{W}}^2
&=&
-2\kappa\int_0^t\big(\widehat{A}\widetilde{Y}^{\e}(s),\widetilde{Y}^{\e}(s)\big)_{\mathbb{W}}dt\nonumber\\
& &
+2\sqrt{\e}\int_0^t\big(\widehat{G}(\widetilde{X}^{\e}(s),s),\widetilde{Y}^{\e}(s)\big)_{\mathbb{W}}dW(s)
+\e\int_0^t\|\widehat{G}(\widetilde{X}^{\e}(s),s)\|_{\mathbb{W}}^2ds\nonumber\\
& &
+2\e\int_0^t\int_{\mathbb{Z}}\big(\widehat{\sigma}(s,\widetilde{X}^{\e}(s-),z),
\widetilde{Y}^{\e}(s-)\big)_{\mathbb{W}}{\widetilde{N}}^{\e^{-1}\varphi_\e}(dzds)\nonumber\\
& &
+{\e}^2\int_0^t\int_{\mathbb{Z}}\|\widehat{\sigma}(s,\widetilde{X}^{\e}(s-),z)\|_{\mathbb{W}}^2N^{\e^{-1}\varphi_\e}(dzds).
\end{eqnarray}
Similar to (\ref{Estimation 2W eq 4}) and (\ref{LDP Estimate eq 3}), we get
\begin{eqnarray}\label{LDP Estimate 4W eq 4}
& &\|\widetilde{Y}^{\e}(t)\|_{\mathbb{W}}^2+\frac{2\kappa}{\alpha}\int_0^t\|\widetilde{Y}^{\e}(s)\|_{\mathbb{W}}^2ds\nonumber\\
&=&
\frac{2\kappa}{\alpha}\int_0^t
\Big(curl\big(\widetilde{Y}^{\e}(s)\big),curl\big(\widetilde{Y}^{\e}(s)-\alpha\Delta\widetilde{Y}^{\e}(s)\big)\Big)_{\mathbb{W}}ds\nonumber\\
& &
+2\sqrt{\e}\int_0^t\big(\widehat{G}(\widetilde{X}^{\e}(s),s),\widetilde{Y}^{\e}(s)\big)_{\mathbb{W}}dW(s)
+\e\int_0^t\|\widehat{G}(\widetilde{X}^{\e}(s),s)\|_{\mathbb{W}}^2ds\nonumber\\
& &
+2\e\int_0^t\int_{\mathbb{Z}}\big(\widehat{\sigma}(s,\widetilde{X}^{\e}(s-),z),
\widetilde{Y}^{\e}(s-)\big)_{\mathbb{W}}{\widetilde{N}}^{\e^{-1}\varphi_\e}(dzds)\nonumber\\
& &
+{\e}^2\int_0^t\int_{\mathbb{Z}}\|\widehat{\sigma}(s,\widetilde{X}^{\e}(s-),z)\|_{\mathbb{W}}^2N^{\e^{-1}\varphi_\e}(dzds).
\end{eqnarray}

Taking the sup over $t\leq T\wedge\tau_M$, then taking expectations we get
\begin{eqnarray}\label{LDP Estimate 4W eq 5}
& &E\Big[\sup_{t\in[0,T\wedge\tau_M]}\|\widetilde{Y}^{\e}(t)\|_{\mathbb{W}}^2\Big]
+\frac{2\kappa}{\alpha} E\Big[\int_0^{T\wedge\tau_M}\|\widetilde{Y}^{\e}(s)\|_{\mathbb{W}}^2ds\Big]\nonumber\\
&=&
\frac{2\kappa}{\alpha}E\Big[\int_0^{T\wedge\tau_M}\Big(curl\big(\widetilde{Y}^{n,\e}(s)\big),
curl\big(\widetilde{Y}^{\e}(s)-\alpha\Delta\widetilde{Y}^{\e}(s)\big)\Big)ds\Big]\nonumber\\
& &
+2\sqrt{\e}E\Big[\sup_{t\in[0,T\wedge\tau_M]}\big|\int_0^{t}\big(\widehat{G}(\widetilde{X}^\e(s),s),\widetilde{Y}^\e(s)\big)dW(s)_{\mathbb{W}}\big|\Big]
+\e E\Big[\int_0^{T\wedge\tau_M}\|\widehat{G}(\widetilde{X}^\e(s),s)\|_\mathbb{W}^2ds\Big]\nonumber\\
& &
+2\e E\Big[\sup_{t\in[0,T\wedge\tau_M]}\big|\int_0^t\int_{\mathbb{Z}}\big(\widehat{\sigma}(s,\widetilde{X}^\e(s-),z),
\widetilde{Y}^\e(s-)\big)_{\mathbb{W}}{\widetilde{N}}^{\e^{-1}\varphi_\e}(dzds)\big|\Big]\nonumber\\
& &
+\e^2E\Big[\int_0^{T\wedge\tau_M}\int_{\mathbb{Z}}\|\widehat{\sigma}(s,\widetilde{X}^\e(s-),z)\|_\mathbb{W}^2N^{\e^{-1}\varphi_\e}(dzds)\Big]\nonumber\\
&:=&
I_1+I_2+I_3+I_4+I_5.
\end{eqnarray}

By Condition \ref{condition 1} and Lemma \ref{H-lemma 1}, we get
\begin{eqnarray}\label{LDP Estimate 4W eq 6}
I_3
&\leq&
\e CE\Big[\sup_{s\in[0,T]}\|\widetilde{X}^\e(s)\|_\mathbb{W}^2\Big],
\end{eqnarray}
and
\begin{eqnarray}\label{LDP Estimate 4W eq 7}
I_5
&\leq&
C\e E\Big[\int_0^{T\wedge\tau_M}\int_{\mathbb{Z}}\|{\sigma}(t,z)\|_{0,\mathbb{V}}^2
\big(1+\sup_{s\in[0,t]}\|\widetilde{X}^\e(s)\|_\mathbb{V}^2\big)\varphi_\e(s,z)\nu(dz)dt\Big]\nonumber\\
&\leq&
{\e}C\cdot C_{0,2}^m\Big(1+E\Big[\sup_{s\in[0,T]}\|\widetilde{X}^\e(s)\|_\mathbb{W}^2\Big]\Big).
\end{eqnarray}
\indent By (\ref{curl}), we have
\begin{eqnarray}\label{LDP Estimate 4W eq 6}
I_1
&\leq&
CE\Big[\int_0^{T\wedge\tau_M}\|\widetilde{Y}^{\e}(s)\|_{\mathbb{V}}\|\widetilde{Y}^{\e}(s)\|_{\mathbb{W}}ds\Big]\nonumber\\
&\leq&
C\int_0^TE\Big[\sup_{s\in[0,t\wedge\tau_M]}\|\widetilde{Y}^{\e}(s)\|_{\mathbb{W}}^2\Big]dt.
\end{eqnarray}
\indent By B-D-G inequality and Young inequality and Lemma \ref{H-lemma 1}, we get
\begin{eqnarray}\label{LDP Estimate 4W eq 7}
I_2
&\leq&
2\sqrt{\e}E\Big[\int_0^{T\wedge\tau_M}
\big|\big(\widehat{G}(\widetilde{X}^{\e}(s),s),\widetilde{Y}^{\e}(s)\big)_{\mathbb{W}}\big|^2ds\Big]^{\frac{1}{2}}\nonumber\\
&\leq&
\sqrt{\e}C E\Big[\sup_{t\in[0,T\wedge\tau_M]}\|\widetilde{Y}^{\e}(t)\|_\mathbb{W}
\big(\int_0^{T\wedge\tau_M}\|\widehat{G}(\widetilde{X}^{\e}(s),s)\|_\mathbb{W}^2ds\big)^{\frac{1}{2}}\Big]\nonumber\\
&\leq&
\frac{1}{4} E\Big[\sup_{t\in[0,T\wedge\tau_M]}\|\widetilde{Y}^{n,\e}(t)\|_\mathbb{W}^2\Big]+
\e TC E\Big[\sup_{s\in[0,T\wedge\tau_M]}\|\widetilde{X}^{\e}(s)\|_\mathbb{W}^2\Big],
\end{eqnarray}
and
\begin{eqnarray}\label{LDP Estimate 4W eq 8}
I_4
&\leq&
\e C E\Big[\int_0^{T\wedge\tau_M}\int_{\mathbb{Z}}\big(\widehat{\sigma}(s,\widetilde{X}^\e(s-),z),\widetilde{Y}^\e(s-)\big)_\mathbb{W}^2
N^{\e^{-1}\varphi_\e}(dzds)\big|\Big]^{\frac{1}{2}}\nonumber\\
&\leq&
\e C E\Big[\sup_{t\in[0,T\wedge\tau_M]}\|\widetilde{Y}^\e(t)\|_\mathbb{W}
\big(\int_0^{T\wedge\tau_M}\int_{\mathbb{Z}}\|\widehat{\sigma}(s,\widetilde{X}^\e(s-),z)\|_\mathbb{W}^2N^{\e^{-1}\varphi_\e}(dzds)\big)^{\frac{1}{2}}\Big]\nonumber\\
&\leq&
\frac{1}{4} E\Big[\sup_{t\in[0,T\wedge\tau_M]}\|\widetilde{Y}^\e(t)\|_\mathbb{W}^2\Big]\nonumber\\
& &\qquad+\e C E\Big[\int_0^{T\wedge\tau_M^n}\int_{\mathbb{Z}}\|\sigma(s,z)\|^2_{0,\mathbb{V}}
\big(1+\|\widetilde{X}^\e(s)\|_\mathbb{V}^2\big)\varphi_\e(s,z)\nu(dz)ds\Big]\nonumber\\
&\leq&
\frac{1}{4} E\Big[\sup_{t\in[0,T\wedge\tau_M]}\|\widetilde{Y}^\e(t)\|_\mathbb{V}^2\Big]
+\e C\cdot C_{0,2}^m\Big(1+E\Big[\sup_{t\in[0,T]}\|\widetilde{X}^\e(s)\|_\mathbb{W}^2\Big]\Big).
\end{eqnarray}
Combining above all inequalities, we get
\begin{eqnarray*}
& &\frac{1}{2}E\Big[\sup_{t\in[0,T\wedge\tau_M]}\|\widetilde{Y}^\e(t)\|_\mathbb{W}^2\Big]
+\frac{2\kappa}{\alpha} E\Big[\int_0^{T\wedge\tau_M}\|\widetilde{Y}^\e(t)\|_\mathbb{W}dt\Big]\nonumber\\
&\leq&
\e C\Big(1+E\Big[\sup_{t\in[0,T]}\|\widetilde{X}^\e(t)\|_\mathbb{W}^2\Big]\Big)
+C\int_0^{T}E\Big[\sup_{s\in[0,t\wedge\tau_M]}\|\widetilde{Y}^\e(s)\|_\mathbb{W}^2\Big]dt.
\end{eqnarray*}
By (\ref{LDP Estimate 1W}) and Gronwall's inequality , then letting $M\rightarrow\infty$, we get (\ref{LDP Estimate 3W}).
\end{proof}

\begin{prp}\label{LDP 2}
Fix $m\in\mathbb{N}$, and let $\phi_\e=(\psi_\e,\varphi_\e),\phi=(\psi,\varphi)\in{{\mathcal{S}}}_1^m\times{{\mathcal{S}}}_2^m$ be such that $\phi_\e$ converges in distribution to $\phi$ as $n\rightarrow\infty$. Then
\begin{eqnarray*}
\mathcal{G}^\e(\sqrt{\e}W(\cdot)+\int_0^\cdot \psi_\e(s)ds,\e N^{{\e}^{-1}\varphi_\e})\text{ converges in distribution to }\mathcal{G}^0(\int_0^\cdot \psi(s)ds,\nu_T^\varphi)\ \ \text{in} \ \ D([0,T],\mathbb{V}).
\end{eqnarray*}
\end{prp}
\begin{proof}

Set $\widetilde{Z}^\e=\widetilde{X}^\e-\widetilde{Y}^\e$, which satisfies
\begin{eqnarray}\label{prop 2 equation 1}
\widetilde{Z}^\e(t)
&=&
X_0-\kappa\int_0^t\widehat{A}\widetilde{Z}^\e(s)ds
-\int_0^t\widehat{B}\big(\widetilde{Z}^\e(s)+\widetilde{Y}^{\e}(s),\widetilde{Z}^\e(s)+\widetilde{Y}^{\e}(s)\big)ds\nonumber\\
& &
+\int_0^t\widehat{F}(\widetilde{Z}^\e(s)+\widetilde{Y}^{\e}(s),s)ds+\int_0^t\widehat{G}(\widetilde{Z}^\e(s)+\widetilde{Y}^{\e}(s),s)\psi_\e(s)ds\nonumber\\
& &
+\int_0^t\int_{\mathbb{Z}}\widehat{\sigma}(s,\widetilde{Z}^\e(s)+\widetilde{Y}^{\e}(s),z)(\varphi_{\e}(s,z)-1)\nu(dz)ds.
\end{eqnarray}
 Set
$$\Pi=\Big(S^m;L^2([0,T],\mathbb{W})\cap D([0,T],\mathbb{V})\Big).$$

\indent Let $\big((\psi,\varphi),0\big)$ be any limit point of the tight family $\big\{\big((\psi_\e,\varphi_\e),\widetilde{Y}^\e\big),\e\in(0,\e_0)\big\}$ in $\Pi$. By Skorokhod representation theorem, there exists a probability space $\big(\Omega^1,\mathcal{F}^1,P^1\big)$ and on this basis, there exist $\Pi$-valued random variables $\big((\psi^1,\varphi^1),0\big)$, $\big\{\big((\psi_\e^1,\varphi_{\e}^1),\widetilde{Y}^{\e}_1\big),\e\in(0,\e_0)\big\}$ such that $\big((\psi_\e^1,\varphi_{\e}^1),\widetilde{Y}^{\e}_1\big)$ \Big(respectively $\big((\psi^1,\varphi^1),0\big)$\Big) has the same law as $\big((\psi_\e,\varphi_\e),\widetilde{Y}^\e\big)$ \Big(respectively $\big((\psi,\varphi),0\big)$\Big) and $\big((\psi_\e^1,\varphi_{\e}^1),\widetilde{Y}^{\e}_1\big)\rightarrow\big((\psi^1,\varphi^1),0\big)$ $P^1$-a.s. in $\Pi$.\\

\indent Since
\begin{eqnarray}\label{prop 2 equation 2}
E^1\Big[\sup_{t\in[0,T]}\|\widetilde{Y}^{\e}_1(t)\|_\mathbb{W}^2\Big]=E\Big[\sup_{t\in[0,T]}\|\widetilde{Y}^{\e}(t)\|_\mathbb{W}^2\Big]\leq\sqrt{\e}C,
\end{eqnarray}
we get\\
(1)$\sup_{t\in[0,T]}\|\widetilde{Y}^{\e}_1(t)\|_\mathbb{W}^2<\infty,\ \ P^1-a.s.$ \\
(2)there exists a sub-sequence $\e_k$ and a subset $\Omega^1_0$ of $\Omega^1$, such that $P^1\big(\Omega^1_0\big)=1$ and $\forall\omega^1\in\Omega^1_0$
\begin{eqnarray}\label{prop 2 equation 3}
\lim_{\e_k\rightarrow0}\sup_{s\in[0,T]}\|\widetilde{Y}^{\e_k}_1(\omega^1,s)\|_\mathbb{W}^2=0.
\end{eqnarray}

\indent Let $\bar{Z}^\e$ be the solution of the following equation,
\begin{eqnarray}\label{prop 2 equation 4}
\bar{Z}^\e(t)
&=&
X_0-\kappa\int_0^t\widehat{A}\bar{Z}^\e(s)ds
-\int_0^t\widehat{B}\big(\bar{Z}^\e(s)+\widetilde{Y}^{\e}_1(s),\bar{Z}^\e(s)+\widetilde{Y}^{\e}_1(s)\big)ds\nonumber\\
& &
+\int_0^t\widehat{F}(\bar{Z}^\e(s)+\widetilde{Y}^{\e}_1(s),s)ds+\int_0^t\widehat{G}(\bar{Z}^\e(s)+\widetilde{Y}^{\e}_1(s),s)\psi_\e^1(s)ds\nonumber\\
& &
+\int_0^t\int_{\mathbb{Z}}\widehat{\sigma}(s,\bar{Z}^\e(s)+\widetilde{Y}^{\e}_1(s),z)(\varphi_{\e}^1(s,z)-1)\nu(dz)ds.
\end{eqnarray}
Keep in mind claim (1), similarly as the proof of Theorem \ref{Theorem skeletons}, we can prove (\ref{prop 2 equation 4}) has a unique solution.\\
\indent Comparing (\ref{prop 2 equation 1}) and (\ref{prop 2 equation 4}), and $\big((\psi_\e^1,\varphi_{\e}^1),\widetilde{Y}^{\e}_1\big)$ has the same law as $\big((\psi_\e,\varphi_\e),\widetilde{Y}^\e\big)$, we can conclude that $\big(\bar{Z}^\e,\widetilde{Y}^{\e}_1\big)$ has the same law as $\big(\widetilde{Z}^\e,\widetilde{Y}^\e\big)$, hence
\begin{eqnarray}\label{prop 2 equation 5}
\bar{Z}^\e+\widetilde{Y}^{\e}_1\stackrel{\mathcal{D}}{=}\widetilde{Z}^\e+\widetilde{Y}^\e.
\end{eqnarray}

By (\ref{prop 2 equation 3}) and a similar argument as in the proof of Proposition \ref{LDP 1}, we can show that
\begin{eqnarray}\label{prop 2 equation 6}
\lim_{\e_k\rightarrow0}\Big[\sup_{t\in[0,T]}\|\bar{Z}^{\e_k}(\omega^1,t)-\widehat{X}(\omega^1,t)\|_\mathbb{V}^2\Big]=0,
\end{eqnarray}
where
\begin{eqnarray*}
\widehat{X}(t)
&=&
X_0-\kappa\int_0^t\widehat{A}\widehat{X}(s)ds
-\int_0^t\widehat{B}\big(\widehat{X}(s),\widehat{X}(s)\big)ds\nonumber\\
& &
+\int_0^t\widehat{F}(\widehat{X}(s),s)ds+\int_0^t\widehat{G}(\widehat{X}(s),s)\psi^1(s)ds\nonumber\\
& &
+\int_0^t\int_{\mathbb{Z}}\widehat{\sigma}(s,\widehat{X}(s),z)(\varphi^1(s,z)-1)\nu(dz)ds.
\end{eqnarray*}
Actually $\widehat{X}$ has the same law of $\mathcal{G}^0(\int_0^\cdot \psi(s)ds,\nu_T^\varphi)$.  
Using (\ref{prop 2 equation 3}) and (\ref{prop 2 equation 6}) to obtain
\begin{eqnarray*}
\lim_{\e_k\rightarrow0}\Big[\sup_{t\in[0,T]}\|\bar{Z}^{\e_k}(t)+\widetilde{Y}^{\e_k}_1(t)-\widehat{X}(t)\|_\mathbb{V}^2\Big]=0.
\end{eqnarray*}
Combining (\ref{prop 2 equation 5}), we yield
\begin{eqnarray*}
\widetilde{X}^\e(\cdot) \text{ converges in distribution to }\mathcal{G}^0(\int_0^\cdot \psi(s)ds,\nu_T^\varphi).
\end{eqnarray*}
The proof is completed.
\end{proof}

\def\refname{ References}

\end{document}